\documentclass[12pt]{amsart}
\usepackage{amsfonts}
\usepackage{mathrsfs}
\usepackage{amscd,amsmath,amssymb,amsfonts}
\usepackage{color}
\usepackage[all]{xy}
\theoremstyle{plain}
\newtheorem{thm}{Theorem}
\newtheorem{lem}[thm]{Lemma}

\newtheorem{prop}[thm]{Proposition}

\theoremstyle{definition}
\newtheorem{defn}[thm]{Definition}

\newtheorem{rmks}[thm]{Remarks}

\newtheorem{nota}[thm]{Notation}

\numberwithin{thm}{section} \numberwithin{equation}{section}

\newcommand{\ga}[2]{\begin{gather}\label{#1}#2 \end{gather}}

\newcommand{\sE}{{\mathcal E}}
\newcommand{\sF}{{\mathcal F}}

\newcommand{\sH}{{\mathcal H}}

\newcommand{\sL}{{\mathcal L}}
\newcommand{\sM}{{\mathcal M}}
\newcommand{\sN}{{\mathcal N}}
\newcommand{\sO}{{\mathcal O}}
\newcommand{\sP}{{\mathcal P}}
\newcommand{\sQ}{{\mathcal Q}}
\newcommand{\sR}{{\mathcal R}}

\newcommand{\sU}{{\mathcal U}}

\newcommand{\sW}{{\mathcal W}}
\newcommand{\sX}{{\mathcal X}}

\newcommand{\Y}{{\mathbb Y}}

\def\wt#1{\widetilde {#1}}

\begin{document}

\title{A finite dimensional proof of Verlinde Formula}
\author{Xiaotao Sun}
\address{Center of Applied Mathematics, School of Mathematics, Tianjin University, No.92 Weijin Road, Tianjin 300072, P. R. China}
\email{xiaotaosun@tju.edu.cn}
\author{Mingshuo Zhou}
\address{Center of Applied Mathematics, School of Mathematics, Tianjin University, No.92 Weijin Road, Tianjin 300072, P. R. China}
\email{zhoumingshuo@amss.ac.cn}
\date{January 29, 2020}
\thanks{Both authors are supported by the National Natural Science Foundation of China No.11831013;
Mingshuo Zhou is also supported by the National Natural Science Foundation of China No.11501154}

\begin{abstract} We prove two recurrence relations among dimensions $$D_g(r,d,\omega):={\rm dim}\,{\rm H}^0(\sU_{C,\,\omega},\Theta_{\sU_{C,\,\omega}})$$ of spaces of generalized theta functions on moduli spaces $\sU_{C,\,\omega}$. By using of these recurrence relations, an explicit formula (Verlinde formula) of $D_g(r,d,\omega)$ is proved (See Theorem \ref{thm4.3}).
\end{abstract}
\maketitle
\begin{quote}

\end{quote}
\section{Introduction}

Let $C$ be a smooth projective curve and $J^d_C$ be the moduli space of rank $1$ vector bundles of degree $d$ on $C$ (i.e. the Jacobian of $C$). It is a classical
theorem that the space ${\rm H}^0(J^d_C, \Theta_{J^d_C})$ of theta functions of order $k$ on $J^d_C$ has dimension $k^g$. A natural question is to find
a formula of dimension for the space ${\rm H}^0(\sU_C, \Theta_{\sU_C})$ of generalized theta functions of order $k$ on
the moduli space $\sU_C$ of semistable vector bundles on $C$ with rank $r$ and degree $d$. It seems impossible for mathematicians to guess such a formula without
the help of Rational Conformal Field Theories (RCFT)(cf. \cite{Ver}).
RCFT is defined to be a functor which associates
a finite-dimensional vector
space $V_C(I, \{\vec a(x)\}_{x\in I})$ to a marked projective curve $(C, I, \{\vec a(x)\}_{x\in I})$ satisfying certain axioms (A0--A4) (cf. \cite{Be3} for the detail). The axioms, in particular the factorization rules (A2 and A4),
can be encoded in a finite-dimensional $\mathbb{Z}$-algebra(the so called fusion ring of the theory). An explicit formula (so called Verlinde formula) for the dimension of $V_C(I, \{\vec a(x)\}_{x\in I})$ can be obtained in terms of the characters of the fusion ring (See Proposition 3.3 of \cite{Be3}).

An important example of RCFT was constructed for a Lie algebra $\mathfrak{g}$ in \cite{TUY} (WZW-models) by associating a space $V_C(\mathfrak{g}, I, \{\vec a(x)\}_{x\in I})$ of conformal blocks to a marked projective curve $(C, I, \{\vec a(x)\}_{x\in I})$. It is this example that relates RCFT to algebraic geometry
when the space of conformal blocks was proved to be the spaces of generalized theta functions on moduli spaces of parabolic $G$-bundles (
${\rm Lie}(G)=\mathfrak{g}$) (cf. \cite{BL}, \cite{Pa} and \cite{Fa}). Then the characters of its fusion ring are determined in terms of representations
of $\mathfrak{g}$  (cf. \cite{Be2} for $\mathfrak{g}=\mathfrak{sl}(r),\,\mathfrak{sp}(r)$ and \cite{Fa} for all classical algebras). Thus an explicit formula
(Verlinde formula) for the dimension of spaces of generalized theta functions is proved. This kind of proof was called \textit{infinite dimensional proof} in \cite{Be}.

Let $\sU_{C,\,\omega}$ be moduli spaces of semistable parabolic bundles
of rank $r$ and degree $d$ on smooth curves $C$ of genus $g\ge 0$ with parabolic structures determined by $\omega=(k,\{\vec n(x),\vec a(x)\}_{x\in I})$.
It is natural to ask if one can prove an explicit formula of $D_g(r,d,\omega):={\rm dim}\,{\rm H}^0(\sU_{C,\,\omega},\Theta_{\sU_{C,\,\omega}})$
without using conformal blocks ? which is called \textit{finite dimensional proof} in \cite{Be}.
There exist some such proofs duo to Bertram, Szenes, Thaddeus, Zagier, Donaldson, Witten, Narasimhan and Ramadas (See \cite{Be} for the survey and references).
As Beauville pointed out in \cite{Be}, all of these proofs only work for $r=2$. As far as we know, a \textit{finite dimensional proof} of Verlinde formula remains open for
$r>2$ (See the comments in \cite{Be} and \cite{Be3}). In fact, even for $r=2$, we know only one such proof duo to Narasimhan and Ramadas, which covers the case of parabolic bundles (See also \cite{DW1} and \cite{DW2} for an analytic proof when $g\ge 2$). When $r=2$, by the result of \cite{NR},  Ramadas proved in \cite{R} a formula of $D_g(2,d,\omega)$ by reducing it to the case of $g=0$ and using a formula of $D_0(2,d,\omega)$. Unfortunately, the formula of $D_0(2,d,\omega)$ was taken from \cite{GK} where its proof is not algebraic.

One of the main results in this article is the following two recurrence relations of $D_g(r,d,\omega)$.
\begin{thm}[See Theorem \ref{thm3.6} and Theorem \ref{thm3.12}]\label{thm1.1} For partitions $g=g_1+g_2$ and $I=I_1\cup I_2$, let
$W_k=\{\,\lambda=(\lambda_1,...,\lambda_r)\,|\, 0=\lambda_r\le\lambda_{r-1}\le\cdots\le\lambda_1\le k\,\}$ and
$$W'_k=\left\{\,\lambda\in W_k\,\,\mid\,\,\left(\sum_{x\in
I_1}\sum^{l_x}_{i=1}d_i(x)r_i(x)+\sum^r_{i=1}\lambda_i\right) \equiv 0({\rm mod}\,\,r)\right\}.$$
Then we have the following recurrence relation
\ga{1.1}{D_g(r,d,\omega)=\sum_{\mu}D_{g-1}(r,d,\omega^{\mu})}
\ga{1.2} {D_g(r,d,\omega)=\sum_{\lambda\in W'_k} D_{g_1}(r, 0,\omega_1^{\lambda})\cdot D_{g_2}(r,d,\omega_2^{\lambda}),}
where $\mu=(\mu_1,\cdots,\mu_r)$ runs through $0\le\mu_r\le\cdots\le
\mu_1< k$ and $\omega^{\mu}$, $\omega_1^{\lambda}$, $\omega_2^{\lambda}$ are explicitly determined by $\mu$ and $\lambda$.
\end{thm}

The idea is to consider a family $\{C_t\}_{t\in \Delta}$ of curves degenerating to a curve $C_0$ with exactly one node. A factorization theorem
$${\rm H}^0(\sU_{C_0,\,\omega_0},\Theta_{\sU_{C_0,\,\omega_0}})=\bigoplus_{\mu}{\rm H}^0(\sU_{\wt C_0,\,\omega^{\mu}_0},\Theta_{\sU_{\wt C_0,\,\omega^{\mu}_0}})$$ for irreducible $C_0$ was proved in \cite{Su1} (See \cite{NR} for $r= 2$) where $\wt C_0$ is the normalization of $C_0$. Finally one has to show that ${\rm dim}\,{\rm H}^0(\sU_{C_t,\,\omega_t},\Theta_{\sU_{C_t,\,\omega_t}})$ is independent of $t\in\Delta$,
which follows clearly ${\rm H}^1(\sU_{C_t,\,\omega_t},\Theta_{\sU_{C_t,\,\omega_t}})=0$. When $C_0$ is irreducible, the vanishing theorem
was proved under assumption that $g\ge 3$ (See \cite{NR} and \cite{Su1}).  Although we have shown in \cite{Su3} that ${\rm dim}\, {\rm H}^0(\sU_{C_t,\,\omega_t},\Theta_{\sU_{C_t,\,\omega_t}})$ is constant for $t\neq 0$ without using vanishing theorem (cf. Corollary 4.8 of \cite{Su3}), the vanishing theorems
for singular curves $C_0$ are needed in order to show
${\rm dim}\,{\rm H}^0(\sU_{C_t,\,\omega_t},\Theta_{\sU_{C_t,\,\omega_t}})
={\rm dim}\,{\rm H}^0(\sU_{C_0,\,\omega_0},\Theta_{\sU_{C_0,\,\omega_0}})\quad (\forall\,\,t\in\Delta)$. When $r=2$, Ramadas proved ${\rm H}^1(\sU_{C_t,\,\omega_t},\Theta_{\sU_{C_t,\,\omega_t}})=0$ in \cite{R} for $g\ge 0$ and irreducible $C_0$ (thus the
recurrence relation \eqref{1.1} for $r=2$). When $C_0=C_1\cup C_2$ is the union of two smooth curves, a factorization theorem
$${\rm H}^0(\sU_{C_0,\,\omega_0},\Theta_{\sU_{C_0,\,\omega_0}})=\bigoplus_{\mu}{\rm H}^0(\sU_{C_1,\,\omega^{\mu}_1},\Theta_{\sU_{C_1,\,\omega^{\mu}_1}})
\otimes {\rm H}^0(\sU_{C_2,\,\omega^{\mu}_2},\Theta_{\sU_{C_2,\,\omega^{\mu}_2}})$$
was proved in \cite{Su2} for $r\ge 2$. Thus, to prove the recurrence relations \eqref{1.1} and \eqref{1.2}, we need firstly to prove
${\rm H}^1(\sU_{C_0,\,\omega_0},\Theta_{\sU_{C_0,\,\omega_0}})=0$
for both cases of $C_0$ irreducible and reducible. The arguments of \cite{R} seems not work for these general cases, our proof of vanishing theorems uses the main results
of \cite{SZ} that the modulo $p$ reduction of moduli spaces are globally $F$-regular for almost $p$ (such varieties are called of globally $F$-regular type). Moreover, in order to obtain the recurrence relation \eqref{1.2}, we have to study the behaving of $D_g(r,d,\omega)$ under Hecke transformations, which is one of technical parts in this article.

The recurrence relation \eqref{1.1} and \eqref{1.2} reduce the problem to compute $D_0(r,d,\omega)$ with $3$ parabolic points (i.e. $|I|=3$). However, computation of
$D_0(r,d,\omega)$ when $|I|=3$ is rather nontrivial and we are not able to find any reference of such computation. In fact, another technical part of this article is
a computation of $D_0(r,d,\omega)$ when $|I|=3$. By using of the recurrence relation \eqref{1.2} (See Proposition \ref{prop4.8}), it is reduced to the computation of $D_0(r,0,\{\omega_s,\lambda_y,\lambda_z\})$.

We describe briefly content of the article. In Section 2, we recall the notions of globally $F$-regular type varieties and the main results of \cite{SZ}. In Section 3,
we prove firstly the vanishing theorem
\begin{thm} [See Theorem \ref{thm3.1}, Theorem \ref{thm3.3} and Theorem \ref{thm3.4}]\label{thm1.2} When $C$ is smooth, for any ample line bundle $\sL$ on $\sU_{C,\,\omega}$, we have
$${\rm H}^i(\sU_{C,\,\omega},\sL)=0 \quad \forall\,\,i>0.$$
When $C$ is irreducible with at most one node, we have
$$H^1(\sU_{C,\,\omega},\Theta_{\sU_{C,\,\omega}})=0$$
where $\Theta_{\sU_{C,\,\omega}}$ is the theta line bundle. When $C$ is reducible with at most one node, for any ample line bundle $\sL$ on $\sU_{C,\,\omega}$, we have
$$H^1(\sU_{C,\,\omega},\sL)=0.$$
\end{thm}
The key technical part of Section 3 is to prove Theorem \ref{thm1.1}, where the recurrence relation \eqref{1.1} follows Theorem \ref{thm1.2} and
the factorization theorem in \cite{Su1}. But the recurrence relation \eqref{1.2} is obtained by using factorization theorem in \cite{Su2} and Hecke transformation.
In Section 4, by using of recurrence relations \eqref{1.1} and \eqref{1.2}, we give a self-contained exposition of computation of $D_g(r,d,\omega)$ (See Theorem \ref{thm4.3} for detail):

\begin{thm} For any given data $\omega=(k,\{\vec n(x),\vec a(x)\}_{x\in I})$, we have
$$\aligned &D_g(r,d,\omega)=(-1)^{d(r-1)}\left(\frac{k}{r}\right)^g(r(r+k)^{r-1})^{g-1}\\&\sum_{\vec v}
\frac{{\rm exp}\left(2\pi i\left(\frac{d}{r}-\frac{|\omega|}{r(r+k)}\right)\sum_{i=1}^{r}v_i\right)S_{\omega}\left({\rm exp}\,2\pi i\frac{\vec v}{r+k}\right)} {\prod_{i<j}\left(2\sin\,\pi \frac{v_i-v_j}{r+k}\right)^{2(g-1)}}\endaligned$$
where $\vec v=(v_1,v_2,\ldots,v_r)$ runs through the integers $$0=v_r<\cdots <v_2<v_1<r+k.$$
\end{thm}
One of the key technical results in Section 4 is the computation of $D_0(r,0,\{\omega_s,\lambda_y,\lambda_z\})$ in Lemma \ref{lem4.7} (its proof spends 6 pages).

\

{\it Acknowledegements:} Xiaotao Sun would like to thank T. R. Ramadas for helpful discussions (by emails) about the computation of $D_g(r,d,\omega)$ when $g=0$
and suggestion of the reference \cite{GK}. He also would like to thank M. S. Narasimhan who encourages him consistently to give a purely
algebro-geometric proof of Verlinde formula.

\section{Globally $F$-regular type of moduli spaces }

Let $X$ be a variety over a perfect field $k$ of $char (k)=p>0$, $$F:X\to X$$
be the Frobenius map and $F^e:X\to X$
be the e-th iterate of Frobenius map. When $X$ is normal, for any (weil) divisor $D\in Div(X)$,
$$\sO_X(D)(V)=\{\,f\in K(X)\,|\, div_V(f)+D|_V\ge 0\,\},\quad \forall\,\,V\subset X$$
is a reflexive subsheaf of constant sheaf $K=K(X)$. In fact, we have
$$\sO_X(D)=j_*\sO_{X^{sm.}}(D)$$
where $j:X^{sm.}\hookrightarrow X$ is the open set of smooth points, and $\sO_X(D)$ is an invertible sheaf
if and only if $D$ is a Cartier divisor.

\begin{defn}\label{defn2.1} A normal variety $X$ over a perfect field is called
\emph{stably Frobenius $D$-split} if $\sO_X\to F^e_*\sO_X(D)$ is split for some $e>0$.
$X$ is called \emph{globally F-regular} if $X$ is stably Frobenius $D$-split
for any effective divisor $D$. A variety $X$ is called \emph{Frobenius split} if $\sO_X\to F_*\sO_X$ is split. In particular,
any \emph{globally F-regular} variety is Frobenius split.
\end{defn}

For any scheme $X$ of finite type over a field $K$ of
characteristic zero, there is a
finitely generated $\mathbb{Z}$-algebra $A\subset K$ and an $A$-flat
scheme $$X_A\to S={\rm Spec}(A)$$ such that $X_K=X_A\times_S{\rm
Spec}(K)\cong X$. $X_A\to S={\rm Spec}(A)$ is called an integral model of $X/K$, and a closed fiber
$X_s=X_A\times_S{\rm Spec}(\overline{k(s)})$ is
called "\textbf{modulo $p$ reduction of $X$}" where $p={\rm
char}(k(s))>0$.

\begin{defn}\label{defn2.2} A variety $X$ over a field of characteristic zero is said of \emph{globally F-regular type} (resp.\emph{ F-split type}) if its "\textbf{modulo $p$ reduction of $X$}"  are globally F-regular (resp. \emph{F-split })
for a dense set of $p$.
\end{defn}

An equivalent definition of \emph{globally F-regular type} for a projective variety $X$ is that its modulo $p$ reductions (for a dense set of $p$) are stably Frobenius $D$-split along any effective Cartier divisor $D$, which do not require normality of its modulo $p$ reductions prior to the definition.
Projective varieties of \emph{globally F-regular type} have many nice properties and a good vanishing theorem of cohomology.

\begin{thm}[Corollary 5.3 and Corollary 5.5 of \cite{Sm}]\label{thm2.3} Let $X$ be a projective variety over a field of characteristic zero. If $X$ is of globally F-regular type,
then we have \begin{itemize} \item [(1)] $X$ is normal, Cohen-Macaulay with rational singularities. If $X$ is $\mathbb{Q}$-Gorenstein, then $X$ has log terminal singularities.
\item [(2)] For any nef line bundle $\sL$ on $X$, we have $H^i(X,\sL)=0$ when $i>0$. In particular, $H^i(X,\sO_X)=0$ whenever $i>0$.
\end{itemize}
\end{thm}
In \cite{SZ}, we have proved that moduli spaces of parabolic bundles and generalized parabolic sheaves with a fixed determinant on a smooth curve are of globally $F$-regular type. To state it, we recall firstly the notions of moduli spaces of parabolic bundles and generalized parabolic sheaves.

Let $C$ be an irreducible projective curve of genus $g\ge 0$ over an
algebraically closed field $K$ of characteristic zero, which has at most
one node $x_0\in C$. Let $I$ be a finite set of smooth points of $C$, and
$E$ be a coherent sheaf of rank $r$ and degree $d$ on $C$ (the rank
$r(E)$ is defined to be dimension of $E_{\xi}$ at generic point
$\xi\in C$, and $d=\chi(E)-r(1-g)$).

\begin{defn}\label{defn2.4} By a quasi-parabolic structure of $E$ at a
smooth point $x\in C$, we mean a choice of flag of quotients
$$E_x=Q_{l_x+1}(E)_x\twoheadrightarrow
Q_{l_x}(E)_x\twoheadrightarrow\cdots\cdots\twoheadrightarrow
Q_1(E)_x\twoheadrightarrow Q_0(E)_x=0$$ of the fibre $E_x$, $n_i(x)={\rm
dim}(ker\{Q_i(E)_x\twoheadrightarrow Q_{i-1}(E)_x\})$ ($1\le i\le l_x+1$)
are called type of the flags. If, in addition, a sequence of integers
$$0\leq a_1(x)<a_2(x)<\cdots
<a_{l_x+1}(x)< k$$ are given, we call that $E$ has a parabolic
structure of type $$\vec n(x)=(n_1(x),n_2(x),\cdots,n_{l_x+1}(x))$$ and
weight $\vec a(x)=(a_1(x),a_2(x),\cdots,a_{l_x+1}(x))$ at $x\in C$.
\end{defn}

\begin{defn}\label{defn2.5} For any subsheaf $F\subset E$, let $Q_i(E)_x^F\subset
Q_i(E)_x$ be the image of $F$ and $n_i^F={\rm
dim}(ker\{Q_i(E)_x^F\twoheadrightarrow Q_{i-1}(E)_x^F\})$. Let
$${\rm par}\chi(E):=\chi(E)+\frac{1}{k}\sum_{x\in
I}\sum^{l_x+1}_{i=1}a_i(x)n_i(x),$$
$${\rm par}\chi(F):=\chi(F)+\frac{1}{k}\sum_{x\in
I}\sum^{l_x+1}_{i=1}a_i(x)n^F_i(x).$$
Then $E$ is called semistable (resp., stable) for $\omega=(k, \{\vec n(x),\,\,\vec a(x)\}_{x\in I})$ if for any
nontrivial $E'\subset E$ such that $E/E'$ is torsion free,
one has
$${\rm par}\chi(E')\leq
\frac{{\rm par}\chi(E)}{r}\cdot r(E')\,\,(\text{resp., }<).$$
\end{defn}

\begin{thm}[Theorem 2.13 of \cite{Su3}]\label{thm2.6} There
exists a seminormal projective variety
$\sU_{C,\,\omega}:=\sU_C(r,d,
\{k,\vec n(x),\vec a(x)\}_{x\in I}),$ which is the coarse moduli
space of $s$-equivalence classes of semistable parabolic sheaves $E$
of rank $r$ and $\chi(E)=\chi=d+r(1-g)$ with parabolic structures of type
$\{\vec n(x)\}_{x\in I}$ and weights $\{\vec a(x)\}_{x\in I}$ at
points $\{x\}_{x\in I}$. If $C$ is smooth, then it is normal, with
only rational singularities.
\end{thm}

Recall the construction of $\sU_{C,\,\omega}=\sU_C(r,d,\omega)$. Fix a line bundle $\sO(1)=\sO_C(c\cdot y)$ on $C$ of ${\rm deg}(\sO(1))=c$, let
$\chi=d+r(1-g)$, $P$ denote the polynomial $P(m)=crm+\chi$,
$\sO_C(-N)=\sO(1)^{-N}$ and $V=\Bbb C^{P(N)}$. Let $\bold Q$ be the Quot scheme of quotients $V\otimes\sO_{C}(-N)\to F\to 0$ (of rank
$r$ and degree $d$) on $C$. Thus there is on $C\times\bold Q$ a universal quotient
$$V\otimes\sO_{C\times\bold Q}(-N)\to \sF\to 0.$$
Let $\sF_x=\sF|_{\{x\}\times\bold Q}$ and $Flag_{\vec n(x)}(\sF_x)\to\bold Q$ be the relative flag scheme of type $\vec n(x)$. Let
$$\sR=\underset{x\in I}{\times_{\bold Q}}Flag_{\vec n(x)}(\sF_x)\to \bold Q,$$
on which reductive group ${\rm SL}(V)$ acts. The data $\omega=(k, \{\vec n(x),\,\,\vec a(x)\}_{x\in I})$, more precisely, the
weight $(k,\{\vec a(x)\}_{x\in I})$ determines a polarisation
$$\Theta_{\sR,\omega}=({\rm
det}R\pi_{\sR}\sE)^{-k}\otimes\bigotimes_{x\in I}
\lbrace\bigotimes^{l_x}_{i=1} {\rm det}(\sQ_{\{x\}\times
\sR,i})^{d_i(x)}\rbrace\otimes\bigotimes_q{\rm
det}(\sE_y)^{\ell}$$
on $\sR$ such that the open set $\sR^{ss}_{\omega}$ (resp. $\sR^s_{\omega}$) of
GIT semistable (resp. GIT stable) points are precisely the set of semistable (resp. stable) parabolic sheaves on $C$ (see \cite{Su3}), where $\sE$ is the pullback of $\sF$ (under
 $C\times\sR\to C\times \bold Q$), ${\rm
det}R\pi_{\sR}\sE$ is determinant line bundle of cohomology,
$$\sE_x=\sQ_{\{x\}\times \sR,l_x+1}\twoheadrightarrow\sQ_{\{x\}\times \sR,l_x}\twoheadrightarrow \sQ_{\{x\}\times \sR,l_x-1}
\twoheadrightarrow\cdots\twoheadrightarrow \sQ_{\{x\}\times
\sR,1}\twoheadrightarrow0$$ are universal quotients on $\sR$ of type $\vec n(x)$, $d_i(x)=a_{i+1}(x)-a_i(x)$ and
$$\ell:=\frac{k\chi-\sum_{x\in I}\sum^{l_x}_{i=1}d_i(x)r_i(x)}{r}.$$
Then $\sU_{C,\,\omega}$ is the GIT quotient $\sR^{ss}_{\omega}\xrightarrow{\psi} \sU_{C,\,\omega}:=\sU_C(r,d, \omega)$ and $\Theta_{\sR^{ss},\omega}$
descends to an ample line bundle $\Theta_{\sU_{C,\,\omega}}$ on $\sU_{C,\,\omega}$ when $\ell$ is an integer.

\begin{defn}\label{defn2.7} When $C$ is a smooth projective curve, let
$${\rm Det}: \sU_{C,\,\omega}\to J^d_C,\quad E\mapsto {\rm det}(E):=\bigwedge^rE$$
be the determinant map. Then, for any $L\in J^d_C$, the fiber
$${\rm Det}^{-1}(L):=\sU_{C,\,\omega}^L$$ is called moduli space of parabolic bundles with a fixed determinant.
\end{defn}
\begin{thm}[Theorem 3.7 of \cite{SZ}]\label{thm2.8} The moduli spaces $\sU_{C,\,\omega}^L$ are of globally F-regular type. If $J^0_C$ of $C$ is of F-split type,
so is $\sU_{C,\,\omega}$.
\end{thm}

When $C$ is irreducible with one node $x_0\in C\setminus I$, let $\pi:\wt C\to C$ be the normalization and $\pi^{-1}(x_0)=\{x_1,\,x_2\}\subset \wt C$. Then the
normalization $\sP_{\omega}$ of $\sU_{C,\,\omega}$ is the moduli space of generalized parabolic sheaves on $\wt C$. A \emph{generalized parabolic sheaf} (GPS) $(E,Q)$ of rank $r$ and degree $d$ on $\wt C$ consists of
a sheaf $E$ of degree $d$ on $\wt C$, torsion free of rank $r$ outside
$\{x_1,x_2\}$ with parabolic structures at the points of $I$ and an $r$-dimensional quotient
$E_{x_1}\oplus E_{x_2}\xrightarrow{q} Q\to 0.$

\begin{defn}\label{defn2.9} A GPS $(E,Q)$ on an irreducible smooth curve $\wt C$ is called \emph{semistable} (resp.,
\emph{stable}), if for every nontrivial subsheaf $E'\subset E$ such that
$E/E'$ is torsion free outside $\{x_1,x_2\},$ we have
$$par\chi(E')-dim(Q^{E'})\leq
r(E')\cdot\frac{par\chi(E)-dim(Q)}{r(E)} \,\quad (\text{resp.,
$<$}),$$ where $Q^{E'}=q(E'_{x_1}\oplus E'_{x_2})\subset Q.$
\end{defn}

\begin{thm}[Theorem 2.24 of \cite{Su3}]\label{thm2.10} For any $\omega=(k, \{\vec n(x),\,\,\vec a(x)\}_{x\in I})$, there exists
a normal projective variety $\sP_{\omega}$ with at most rational singularities, which is the
coarse  moduli space of $s$-equivalence classes of semi-stable GPS on $\wt C$ with parabolic structures
at the points of $I$ given by the data $\omega$.
\end{thm}

Recall the construction of $\sP_{\omega}$. Let $Grass_r(\sF_{x_1}\oplus\sF_{x_2})\to \bold Q$ and
$$\widetilde{\sR}=Grass_r(\sF_{x_1}\oplus\sF_{x_2})\times_{\bold
Q}\sR\xrightarrow{\rho} \sR.$$  $\omega=(k, \{\vec n(x),\,\,\vec a(x)\}_{x\in I})$ determines
a polarization, which linearizes the ${\rm SL}(V)$-action on $\wt\sR$, such that the open set $\wt{\sR}^{ss}_{\omega}$ (resp. $\wt{\sR}^s_{\omega}$) of
GIT semistable (resp. GIT stable) points are precisely the set of semistable (resp. stable) GPS on $\wt C$ (see \cite{Su3}). Then $\sP_{\omega}$ is the GIT quotient
\ga{2.1} {\wt{\sR}^{ss}_{\omega}\xrightarrow{\psi}\wt{\sR}^{ss}_{\omega}//{\rm SL}(V):=\sP_{\omega}.}

\begin{nota}\label{nota2.11} Let $\sH\subset\wt{\sR}$ be the open subscheme
parametrising the generalised parabolic sheaves $E=(E,E_{x_1} \oplus
E_{x_2}\xrightarrow{q}Q)$ satisfying \begin{itemize}
\item [(1)] the torsion ${\rm Tor}\,E$ of $E$ is
supported on $\{x_1,x_2\}$ and $$q:({\rm Tor}\,E)_{x_1}\oplus ({\rm
Tor}\,E)_{x_2}\hookrightarrow Q$$
\item [(2)] if $N$ is large enough, then
$H^1(E(N)(-x-x_1-x_2))=0$ for all $E$ and $x\in \wt C$.
\end{itemize}
\end{nota}
Then $\sH$ is reduced, normal, Gorenstein with at most rational singularities (see Proposition 3.2 and Remark 3.1 of \cite{Su1}). Moreover,
for any data $\omega$, we have $\wt{\sR}^{ss}_{\omega}\subset \sH$ and, by Lemma 5.7 of \cite{Su1}, there is a morphism ${\rm Det}_{\sH}:\sH\to J^d_{\wt C}$
which extends determinant morphism on open set $\wt\sR_F\subset\sH$ of locally free sheaves, and induces a flat morphism
\ga{2.2} {{\rm Det}: \sP_{\omega}\to J^d_{\wt C}.}

\begin{nota}\label{nota2.12} For $L\in J^d_{\wt C}$, let $\sH^L={\rm Det}^{-1}(L)\subset\sH$,
$$\wt\sR_F^L={\rm Det}^{-1}(L)\subset\wt\sR_F, \quad (\wt\sR_{\omega}^{ss})^L={\rm Det}^{-1}(L)\subset \wt\sR_{\omega}^{ss}.$$
Then $\sP^L_{\omega}={\rm Det}^{-1}(L)\subset \sP_{\omega}$ is the GIT quotient
$$(\wt\sR_{\omega}^{ss})^L\xrightarrow{\psi}\sP^L_{\omega}=(\wt\sR_{\omega}^{ss})^L//{\rm SL}(V).$$
\end{nota}

\begin{thm}[Theorem 4.7 of \cite{SZ}]\label{thm2.13} For any $\omega=(k, \{\vec n(x),\,\vec a(x)\}_{x\in I})$, the moduli space
$\sP^L_{\omega}$ is of globally F-regular type.
\end{thm}

When $C=C_1\cup C_2$ is reducible with two smooth irreducible
components $C_1$ and $C_2$ of genus $g_1$ and $g_2$ meeting at only
one point $x_0$ (which is the only node of $C$), we fix an ample
line bundle $\sO(1)$ of degree $c$ on $C$ such that
$deg(\sO(1)|_{C_i})=c_i>0$ ($i=1,2$). For any coherent sheaf $E$,
$P(E,n):=\chi(E(n))$ denotes its Hilbert polynomial, which has degree
$1$. We define the rank of $E$ to be
$$r(E):=\frac{1}{deg(\sO(1))}\cdot \lim \limits_{n\to\infty}\frac{P(E,n)}
{n}.$$ Let $r_i$ denote the rank of the restriction of $E$ to $C_i$
($i=1,2$), then
$$P(E,n)=(c_1r_1+c_2r_2)n+\chi(E),\quad r(E)=
\frac{c_1}{c_1+c_2}r_1+\frac{c_2}{c_1+c_2}r_2.$$ We say that $E$ is
of rank $r$ on $X$ if $r_1=r_2=r$, otherwise it will be said of rank
$(r_1,r_2)$.
Fix a finite set $I=I_1\cup I_2$ of smooth points on $C$, where
$I_i=\{x\in I\,|\,x\in C_i\}$ ($i=1,2$), and parabolic data $\omega=\{k,\vec
n(x),\vec a(x)\}_{x\in I}$ with
$$\ell:=\frac{k\chi-\sum_{x\in I}\sum^{l_x}_{i=1}d_i(x)r_i(x)}{r}$$
(recall $d_i(x)=a_{i+1}(x)-a_i(x)$, $r_i(x)=n_1(x)+\cdots+n_i(x)$). Let
\ga{2.3}{n^{\omega}_j=\frac{1}{k}\left(r\frac{c_j}{c_1+c_2}\ell+\sum_{x\in
I_j}\sum^{l_x}_{i=1}d_i(x)r_i(x)\right)\,\,\,(j=1,\,\,2).}

\begin{defn}\label{defn2.14} For any coherent sheaf $F$ of rank $(r_1,r_2)$, let
$$m(F):= \frac{r(F)-r_1}{k}\sum_{x\in I_1}a_{l_x+1}(x)+
\frac{r(F)-r_2}{k}\sum_{x\in I_2}a_{l_x+1}(x),$$ the modified
parabolic Euler characteristic and slop of $F$ are
$${\rm par}\chi_m(F):={\rm par}\chi(F)+m(F),\quad {\rm par}\mu_m(F):=\frac{{\rm par}\chi_m(F)}{r(F)}.$$
A parabolic sheaf $E$ is called semistable (resp. stable) if, for
any subsheaf $F\subset E$ such $E/F$ is torsion free, one has, with
the induced parabolic structure,
$${\rm par}\chi_m(F)\le \frac{{\rm par}\chi_m(E)}{r(E)}r(F)\quad (resp.<).$$
\end{defn}

\begin{thm}[Theorem 1.1 of \cite{Su2} or Theorem 2.14 of \cite{Su3}]\label{thm2.15} There
exists a reduced, seminormal projective scheme
$$\sU_C:=\sU_C(r,d,\sO(1),
\{k,\vec n(x),\vec a(x)\}_{x\in I_1\cup I_2})$$ which is the coarse
moduli space of $s$-equivalence classes of semistable parabolic
sheaves $E$ of rank $r$ and $\chi(E)=\chi=d+r(1-g)$ with parabolic structures
of type $\{\vec n(x)\}_{x\in I}$ and weights $\{\vec a(x)\}_{x\in
I}$ at points $\{x\}_{x\in I}$. The moduli space $\sU_C$ has at most
$r+1$ irreducible components.
\end{thm}

The normalization of $\sU_C$ is a moduli space of semistable GPS on $\wt C=C_1\bigsqcup C_2$ with parabolic structures at points $x\in I$. Recall

\begin{defn}\label{defn2.16} A GPS $(E,E_{x_1}\oplus E_{x_2}\xrightarrow{q}Q)$ is called semistable (resp.,
stable), if for every nontrivial subsheaf $E'\subset E$ such that
$E/E'$ is torsion free outside $\{x_1,x_2\},$ we have, with the
induced parabolic structures at points $\{x\}_{x\in I}$,
$$par\chi_m(E')-dim(Q^{E'})\leq
r(E')\cdot\frac{par\chi_m(E)-dim(Q)}{r(E)} \,\quad (\text{resp.,
$<$}),$$ where $Q^{E'}=q(E'_{x_1}\oplus E'_{x_2})\subset Q.$
\end{defn}

\begin{thm}[Theorem 2.1 of \cite{Su2} or Theorem 2.26 of \cite{Su3}]\label{thm2.17} For any data $\omega=(\{k,\vec n(x),\,\,\vec a(x)\}_{x\in I_1\cup I_2},\sO(1))$,
the coarse  moduli space $\sP_{\omega}$ of $s$-equivalence classes of semi-stable GPS on $\wt C$ with parabolic structures
at the points of $I$ given by the data $\omega$ is a disjoint union of at most $r+1$ irreducible, normal projective varieties $\sP_{\chi_1,\chi_2}$
( $\chi_1+\chi_2=\chi+r$, $n_j^{\omega}\le\chi_j\le n_j^{\omega}+r$) with at most rational singularities.
\end{thm}

For fixed $\chi_1$, $\chi_2$ satisfying $\chi_1+\chi_2=\chi+r$ and $n_j^{\omega}\le\chi_j\le n_j^{\omega}+r$ ($j=1,\,2$), recall the construction
of $\sP_{\omega}=\sP_{\chi_1,\chi_2}$. Let
$$P_i(m)=c_irm+\chi_i,\quad \sW_i=\sO_{C_i}(-N),\quad V_i=\Bbb
C^{P_i(N)}$$ where
$\sO_{C_i}(1)=\sO (1)|_{C_i}$ has degree $c_i$. Consider the Quot schemes $\textbf{Q}_i=Quot(V_i\otimes\sW_i,
P_i)$, the universal quotient
$V_i\otimes\sW_i\to \sF^i\to 0$ on $C_i\times \textbf{Q}_i$ and
the relative flag scheme
$$\sR_i=\underset{x\in I_i}{\times_{\textbf{Q}_i}}
Flag_{\vec n(x)}(\sF^i_x)\to \textbf{Q}_i.$$ Let
$\sF=\sF^1\oplus\sF^2$ denote direct sum of pullbacks of $\sF^1$,
$\sF^2$ on $$\wt C\times
(\textbf{Q}_1\times\textbf{Q}_2)=(C_1\times\textbf{Q}_1)\sqcup(C_2\times\textbf{Q}_2).$$
Let $\sE$ be the pullback of $\sF$ to $\wt
C\times(\sR_1\times\sR_2)$, and
$$\rho:\widetilde{\sR}=Grass_r(\sE_{x_1}\oplus\sE_{x_2})\to\sR=\sR_1\times\sR_2\to
\textbf{Q}=\textbf{Q}_1\times\textbf{Q}_2.$$
For the given $\omega=(\{k,\vec n(x),\,\vec a(x)\}_{x\in I_1\cup I_2},\mathcal{O}(1))$, let
$\wt{\sR}_{\omega}^{ss}$ (resp. $\wt{\sR}_{\omega}^{s}$) denote the open set of GIT semi-stable (resp. GIT stable) points under
action of $G=({\rm GL}(V_1)\times {\rm GL}(V_2))\cap {SL}(V_1\oplus V_2)$ on $\wt\sR$ respect to the polarization determined by $\omega$.
Let $\sH\subset\wt{\sR}$ be the open set defined in Notation \ref{nota2.11}, then for any data $\omega$ we have
$\wt{\sR}_{\omega}^{s}\subset\wt{\sR}_{\omega}^{ss}\subset\sH.$
The moduli space in Theorem \ref{thm2.17} is nothing but the GIT quotient
$$\psi:\wt{\sR}_{\omega}^{ss}\to \sP_{\omega}:=\wt{\sR}_{\omega}^{ss}//G.$$
There exists a morphism $\hat{\rm Det}_{\sH}: \sH\to J^d_{\wt C}=J^{d_1}_{C_1}\times J^{d_2}_{C_2}$, which extends
$$\hat{\rm Det}_{\sH_F}: \sH_F\to J^{d_1}_{C_1}\times J^{d_2}_{C_2},\quad (E,Q)\mapsto ({\rm det}(E|_{C_1}), {\rm det}(E|_{C_2}))$$
on the open set $\sH_F\subset \sH$ of GPB (i.e. GPS $(E,Q)$ with $E$ locally free) and induces a flat determinant morphism
$${\rm Det}_{\sP_{\omega}}:\sP_{\omega}\to J^d_{\wt C}=J^{d_1}_{C_1}\times J^{d_2}_{C_2}$$
(see page 46 of \cite{Su3} for detail). In fact, for any $L\in J^d_{\wt C}=J^{d_1}_{C_1}\times J^{d_2}_{C_2}$, let
\ga{2.4} {\sP_{\omega}^L:={\rm Det}_{\sP_{\omega}}^{-1}(L)\subset \sP_{\omega}}

\begin{thm}[Theorem 4.15 of \cite{SZ}]\label{thm2.18} For any data $$\omega=(\{k,\vec n(x),\,\,\vec a(x)\}_{x\in I_1\cup I_2},\sO(1))$$
and integers $\chi_1$, $\chi_2$ satisfying $\chi_1+\chi_2=\chi+r$, $n_j^{\omega}\le\chi_j\le n_j^{\omega}+r$ ($j=1,\,2$), let $\sP^L_{\omega}$ be
the coarse  moduli space  of $s$-equivalence classes of semi-stable GPS $E=(E_1, E_2)$ on $\wt C$ with fixed determinant $L$, $\chi(E_j)=\chi_j$ and parabolic structures at the points of $I$ given by the data $\omega$. Then $\sP^L_{\omega}$ is of globally $F$-regular type.
\end{thm}

\section{Vanishing Theorems on Moduli spaces and recurrence relations}

In this section, we prove vanishing theorems on moduli spaces of parabolic sheaves on curves with at most one node and establish the
recurrence relations of dimension of generalized theta functions. As an immediate application of globally $F$-regular type of the moduli spaces of parabolic sheaves on a
smooth projective curve $C$, we have the following vanishing theorem

\begin{thm}\label{thm3.1} Let $\sU_{C,\,\omega}$ be the moduli
space of semistable parabolic bundles
of rank $r$ and degree $d$ on a smooth projective curve $C$ with parabolic structures determined by $\omega=(k,\{\vec n(x),\vec a(x)\}_{x\in I})$.
Then $${\rm H}^i(\sU_{C,\,\omega},\sL)=0 \quad \forall\,\,i>0$$
for any ample line bundle $\sL$ on $\sU_{C,\,\omega}$.
\end{thm}

\begin{proof} Let ${\rm Det}: \sU_{C,\,\omega}\to J^d_C$ and $\sU_{C,\,\omega}^L={\rm Det}^{-1}(L)$, then morphism
$$J^0_C\times \sU_{C,\,\omega}^L\to \sU_{C,\,\omega},\quad (L_0, E)\mapsto L_0\otimes E$$
is a $r^{2g}$-fold cover. Then it is enough to show $H^i(J^0_C\times \sU_{C,\,\omega}^L, \sL)=0$ for $i>0$ and any ample
line bundle $\sL$. Since $\sU_{C,\,\omega}^L$ is of globally $F$-regular type, $H^1(\sU_{C,\,\omega}^L, \sO_{\sU_{C,\,\omega}^L})=0$ by Theorem \ref{thm2.3} (2).
Thus $\sL=\sL_1\otimes\sL_2$ where $\sL_1$ (resp. $\sL_2$) is an ample bundle on $J^0_C$ (resp. $\sU_{C,\,\omega}^L$) and
$$H^i(J^0_C\times \sU_{C,\,\omega}^L, \sL)=H^i(J^0_C,\sL_1)\otimes H^0( \sU_{C,\,\omega}^L,\sL_2)=0.$$
\end{proof}

For any irreducible curve $C$ with at most one node $x_0\in C$, there is an algebraic family of ample line bundles $\Theta_{\sU_{C,\,\omega}}$ on
$\sU_{C,\,\omega}$ when
\ga{3.1}{\text{$\ell:=\frac{k\chi-\sum_{x\in
I}\sum^{l_x}_{i=1}d_i(x)r_i(x)}{r}$ is an integer}}
(see Theorem 3.1 of \cite{Su3}). Then Theorem \ref{thm3.1} implies that the number
\ga{3.2} {D_g(r,d,\omega)=dim H^0(\sU_{C,\,\omega},\Theta_{\sU_{C,\,\omega}})}
is independent of $C$, parabolic points $x\in I$ (of course, depending on the number $|I|$ of parabolic points) and the choice of $\Theta_{\sU_{C,\,\omega}}$ in the algebraic family when $C$ is smooth.

When $C$ has one node $x_0\in C$, the moduli spaces $\sU_{C,\,\omega}$ are only seminormal (see Theorem 4.2 of \cite{Su1}) and its normalization
$$\phi:\sP_{\omega}\to \sU_{C,\,\omega}$$
is the coarse  moduli space $\sP_{\omega}$ of $s$-equivalence classes of semi-stable GPS on $\wt C\xrightarrow{\pi} C$ with generalized parabolic structures on $\pi^{-1}(x_0)=x_1+x_2$ and parabolic structures at the points of $\pi^{-1}(I)$ given by the data $\omega$ (see Proposition 2.1 of \cite{Su1}, or Proposition 3.1
of \cite{Su2}).

\begin{lem}[Lemma 5.5 of \cite{Su1}]\label{lem3.2} For any line bundle $\sL$ on $\sU_{C,\,\omega}$,
$$\phi^*:H^1(\sU_{C,\,\omega},\sL)\to H^1(\sP_{\omega},\phi^*\sL)$$
is injective.
\end{lem}

\begin{thm}\label{thm3.3} $H^i(\sP_{\omega},\Theta_{\sP_{\omega}})=0$ ($\forall\,\,i>0$) where $\Theta_{\sP_{\omega}}=\phi^*\Theta_{\sU_{C,\,\omega}}$. In particular, $H^1(\sU_{C,\,\omega}, \Theta_{\sU_{C,\,\omega}})=0.$
\end{thm}

\begin{proof} Let ${\rm Det}:\sP_{\omega}\to J^d_{\wt C}$ be the flat morphism defined in \eqref{2.2}. Then $H^i(\sP_{\omega},\Theta_{\sP_{\omega}})=0$
follows $R^i{\rm Det}_*\Theta_{\sP_{\omega}}=0$ by Theorem \ref{thm2.13} and $H^i(J^d_{\wt C},{\rm Det}_*\Theta_{\sP_{\omega}})=0$ by a decomposition of
${\rm Det}_*\Theta_{\sP_{\omega}}$ (see Remark 4.2 of \cite{Su1} or a more precise version in Lemma 5.2 of \cite{Su3}).
\end{proof}

When $C=C_1\cup C_2$ is a reducible one nodal curve, we have a stronger vanishing theorem on $\sU_{C,\,\omega}$ and $\sP_{\omega}$.

\begin{thm}\label{thm3.4} When $C$ is a reducible one nodal curve with two smooth irreducible components, let $\sP_{\omega}$ be the moduli spaces of semi-stable
GPS on $\wt C$ with parabolic structures determined by $\omega$. Then, for any ample line bundle $\wt\sL$ on $\sP_{\omega}$ and $i>0$, we have $H^i(\sP_{\omega},\wt\sL)=0.$ In particular,
$$H^1(\sU_{C,\,\omega},\, \sL)=0$$
holds for any ample line bundle $\sL$ on $\sU_{C,\,\omega}$.
\end{thm}

\begin{proof} By Lemma \ref{lem3.2}, it is enough to show $H^i(\sP_{\omega},\wt\sL)=0$ for any ample line bundle $\wt\sL$ and $i>0$.

When $C=C_1\cup C_2$, the moduli space
$\sP_{\omega}$ is a disjoint union of
$$\{\sP_{d_1,d_2}\}_{d_1+d_2=d}$$ where $\sP_{d_1,d_2}$ consists of GPS $(E,Q)$ with $d_i=deg(E|_{C_i})$. It is enough to consider
$\sP_{\omega}=\sP_{d_1,d_2}$, thus we have the flat morphism
$${\rm Det}:\sP_{\omega}\to J^d_{\wt C}=J^{d_1}_{C_1}\times
J^{d_2}_{C_2}=J^d_C$$ and $J^0_{\wt C}=J^0_{C_1}\times
J^0_{C_2}=J^0_C$ acts on $\sP_{\omega}$ by
$$((E,Q),\sN)\mapsto (E\otimes\pi^*\sN, Q\otimes \sN_{x_0})$$
where $\pi:\wt C\to C$ is the normalization of $C$.
Let $\sP_{\omega}^L={\rm Det}^{-1}(L)$ and consider the morphism
$f:\sP_{\omega}^L\times J^0_C\to \sP_{\omega}$, which is a finite morphism (see the proof of Lemma 6.6 in \cite{Su3} where we figure out a line bundle
$\Theta$ on $\sP_{\omega}$ such that its pullback $f^*(\Theta)$ is ample). Thus it is enough to prove the vanishing theorem on $\sP_{\omega}^L\times J^0_C$,
which follows the same arguments in the proof of Theorem \ref{thm3.1} by using Theorem \ref{thm2.18}.
\end{proof}

\begin{nota}\label{nota3.5} For $\mu=(\mu_1,\cdots,\mu_r)$  with $0\le\mu_r\le\cdots\le\mu_1<
k,$ let $$\{d_i=\mu_{r_i}-\mu_{r_i+1}\}_{1\le i\le l}$$ be the
subset of nonzero integers in
$\{\mu_i-\mu_{i+1}\}_{i=1,\cdots,r-1}.$ We define
$$ r_i(x_1)=r_i,\quad d_i(x_1)=d_i,\quad l_{x_1}=l,$$
$$ r_i(x_2)=r-r_{l-i+1},\quad d_i(x_2)=d_{l-i+1},\quad l_{x_2}=l,$$ and for $j=1,2$, we set
$$\aligned
\vec a(x_j)&=\left(\mu_r,\mu_r+d_1(x_j),\cdots,\mu_r+
\sum^{l_{x_j}-1}_{i=1}d_i(x_j),\mu_r+\sum^{l_{x_j}}_{i=1}d_i(x_j)\right)\\
\vec n(x_j)&=(r_1(x_j),r_2(x_j)-r_1(x_j),
\cdots,r_{l_{x_j}}(x_j)-r_{l_{x_j}-1}(x_j),r-r_{l_{x_j}}(x_j)).\endaligned$$
\end{nota}

\begin{thm}\label{thm3.6} For any $\omega=(k,\,\{\vec n(x),\vec a(x)\}_{x\in I})$ such that
$$\text{$\ell:=\frac{k\chi-\sum_{x\in
I}\sum^{l_x}_{i=1}d_i(x)r_i(x)}{r}$ is an integer}$$
where $\chi=d+r(1-g)$, let $D_g(r,d,\omega)=dim H^0(\sU_{C,\,\omega},\Theta_{\sU_{C,\,\omega}})$. Then, for any positive integers $c_1$, $c_2$ and partitions $I=I_1\cup I_2$, $g=g_1+g_2$
such that $\ell_j=\frac{c_j\ell}{c_1+c_2}$ ($j=,\,2$) are integers, we have
\ga{3.3}{D_g(r,d,\omega)=\sum_{\mu}D_{g-1}(r,d,\omega^{\mu})}
\ga{3.4}{D_g(r,d,\omega)=\sum_{\mu}D_{g_1}(r,d_1^{\mu},\omega_1^{\mu})\cdot D_{g_2}(r,d_2^{\mu},\omega_2^{\mu})}
where $\mu=(\mu_1,\cdots,\mu_r)$ runs through $0\le\mu_r\le\cdots\le
\mu_1< k$ and $$\omega^{\mu}=(k, \{\vec n(x),\,\vec a(x)\}_{x\in I\cup\{x_1,\,x_2\}}),\,\,\omega_j^{\mu}=(k, \{\vec n(x),\,\vec a(x)\}_{x\in I_j\cup\{x_j\}})$$
with $\vec n(x_j)$, $\vec a(x_j)$ ($j=1,\,2$) determined by $\mu$ (Notation \ref{nota3.5}) and
$$d_1^{\mu}=n^{\omega}_1+\frac{1}{k}\sum^r_{i=1}
\mu_i+r(g_1-1),\quad d^{\mu}_2=n^{\omega}_2+r-\frac{1}{k}\sum^r_{i=1}
\mu_i+r(g_2-1)$$
$$n^{\omega}_j=\frac{1}{k}\left(r\frac{c_j}{c_1+c_2}\ell+\sum_{x\in
I_j}\sum^{l_x}_{i=1}d_i(x)r_i(x)\right)\,\,\,(j=1,\,\,2).$$
\end{thm}

\begin{proof} Consider a flat family of projective $|I|$-pointed curves $\sX\to T$ and a relative ample line bundle $\sO_{\sX}(1)$ of relative degree $c$ such that a fiber $\sX_{t_0}:=X$ ($t_0\in T$) is a connected curve with only one node $x_0\in X$ and $\sX_t$ ($t\in T\setminus\{t_0\}$) are smooth curves with a fiber $\sX_{t_1}=C$ ($t_1\neq t_0$). Then one can associate a family of moduli spaces $\sM\to T$ and a line bundle
$\Theta$ on $\sM$ such that each fiber $\sM_t=\sU_{\sX_t,\,\omega}$ is the moduli space of semi-stable parabolic sheaves on $\sX_t$ and $\Theta|_{\sM_t}=\Theta_{\sU_{\sX_t},\,\omega}$. By degenerating $C$ to an irreducible $X$ and using Theorem \ref{thm3.1} and Theorem \ref{thm3.3}, the recurrence relation \eqref{3.3} is nothing but the Factorization theorem of \cite{Su1}.
If we degenerate $C$ to a reducible curve $X=X_1\cup X_2$ with $g(X_i)=g_i$ and choose the relative ample line bundle $\sO_{\sX}(1)$ such that $c_i={\rm deg}(\sO_{\sX}(1)|_{X_i}$, by using
Theorem \ref{thm3.1} and Theorem \ref{thm3.4}, the recurrence relation \eqref{3.4} is exactly the Factorization theorem of \cite{Su2}.
\end{proof}

In the recurrence relation \eqref{3.4}, the degree $d_1^{\mu}$ varies with $\mu$ and $D_{g_1}(r,d_1^{\mu},\omega_1^{\mu})$ makes sense only when $d_1^{\mu}$ is an integer, which are not
convenient for applications. To remedy it, we are going to study the behavior of $D_g(r,d,\omega)$ under Hecke transformation.

Given a parabolic sheaf $E$ with quasi-parabolic structure
$$E_z=Q_{l_z+1}(E)_z\twoheadrightarrow Q_{l_z}(E)_z\twoheadrightarrow \cdots \cdots \twoheadrightarrow Q_1(E)_z\twoheadrightarrow Q_0(E)_z=0$$
of type $\vec n(z)=(n_1(z),...,n_{l_z+1}(z))$ at $z\in I$ and weights $$ 0= a_1(z) <a_2(z)< \cdots <a_{l_z+1}(z)<k.$$
Let $F_i(E)_z=ker\{E_z\twoheadrightarrow Q_i(E)_z\}$ and  $E'=ker\{E\twoheadrightarrow Q_1(E)_z\}$.
Then, at $z\in I$, $E'$ has a natural quasi-parabolic structure
\ga{4.5} {E'_z\twoheadrightarrow F_1(E)_z\twoheadrightarrow Q_{l_z-1}(E')_z \twoheadrightarrow\cdots\twoheadrightarrow Q_1(E')_z\twoheadrightarrow 0}
of type $\vec n'(z)=(n'_1(z),...,n'_{l_z+1}(z))=(n_2(z),...,n_{l_z+1}(z),n_1(z))$, where
$$Q_i(E')_z\subset Q_{i+1}(E)_z$$ is the image of
$F_1(E)_z$ under $E_z\twoheadrightarrow Q_{i+1}(E)_z$. It is easy to see
$$Q_i(E')_z\cong F_1(E)_z/F_{i+1}(E)_z.$$

\begin{defn}\label{defn3.7} The parabolic sheaf $E'$ with given weight
$$0= a'_1(z)<\cdots < a'_{l_z+1}(z)< k$$
is called Hecke transformation of the parabolic sheaf $E$ at $z\in I$, where $a'_{l_z+1}(z)=k-a_2(z)$
and $a'_i(z)=a_{i+1}(z)-a_2(z)$ for $2\leq i\leq l_z$.
\end{defn}

\begin{lem}\label{lem3.8}  The parabolic bundle $E'$ is semistable (resp., stable) iff $E$ is semistable (resp., stable).
\end{lem}
\begin{proof} $E'$ is defined by the exact sequence of sheaves
$$0\to E'\xrightarrow{i} E\xrightarrow{\delta} Q_1(E)_z\to 0$$
such that $E_z\xrightarrow{\delta_z} Q_1(E)_z$ is the surjective homomorphism
$$E_z=Q_{l_z+1}(E)_z\twoheadrightarrow Q_{l_z}(E)_z\twoheadrightarrow \cdots \cdots \twoheadrightarrow Q_1(E)_z.$$
For any sub-bundle $\sF\subset E$ of rank $r_1$, let $Q_i(E)^{\sF}_z\subset Q_i(E)_z$ be the image of $\sF_z\subset E_z$ under
$E_z\twoheadrightarrow Q_{l_z}(E)_z\twoheadrightarrow \cdots \cdots \twoheadrightarrow Q_i(E)_z$, and the sub-bundle $\sF'\subset E'$ is defined by exact sequence
of sheaves:
$$0\to \sF'\xrightarrow{i} \sF\xrightarrow{\delta} Q_1(E)^{\sF}_z\to 0.$$
Let $Q_{l_z}(E')_z^{\sF'}\subset F_1(E)_z$ and $Q_i(E')_z^{\sF'}\subset Q_i(E')_z$ ($1\le i<l_z$) be the image of $\sF'_z\subset E'_z$
under $E'_z\twoheadrightarrow F_1(E)_z$ and $E'_z\twoheadrightarrow Q_i(E')_z$, which are the surjections in \eqref{4.5}.
Since $Q_{l_z}(E')_z^{\sF'}=ker\{\sF_z\xrightarrow{\delta_z} Q_1(E)_z^{\sF}\}$,
$$ker\{Q_{i}(E')_z^{\sF'}\twoheadrightarrow Q_{i-1}(E')_z^{\sF'}\}=ker\{Q_{i+1}(E)_z^{\sF}\twoheadrightarrow Q_{i}(E)_z^{\sF}\}$$
for $1\le i\le l_z$. In particular, $n_i^{\sF'}(z)=n^{\sF}_{i+1}(z)$,
$n^{\sF'}_{l_z+1}(z)=n^{\sF}_1(z)$,
$${\rm p}ar_{\omega'}\chi(\sF')={\rm p}ar_{\omega}\chi(\sF)-\frac{r_1}{k}a_2(z), \quad {\rm p}ar_{\omega'}\chi(E')={\rm p}ar_{\omega}\chi(E)-\frac{r}{k}a_2(z).$$
Thus ${\rm p}ar_{\omega'}\mu(\sF')-{\rm p}ar_{\omega'}\mu(E')={\rm p}ar_{\omega}\mu(\sF)-{\rm p}ar_{\omega}\mu(E)$, which proves the lemma.
\end{proof}

\begin{lem}\label{lem3.9} For parabolic data $\omega=(k,\{\vec n(x),\vec a(x)\}_{x\in I})$, let
\ga{3.6} {\omega'=(k, \{\vec n(x), \vec a(x)\}_{x\neq z \in I}\cup \{\vec a'(z), \vec n'(z)\})}
where $\vec n'(z)=(n'_1(z),...,n'_{l_z+1}(z))=(n_2(z),...,n_{l_z+1}(z),n_1(z))$,
$$\vec a'(z)=(0, a_2'(z), ..., a'_{l_z+1}(z)),\quad a'_{l_z+1}(z)=k-a_2(z)+a_1(z)$$
and $a'_i(z)=a_{i+1}(z)-a_2(z)+a_1(z)$ for $2\le i\le l_z$. Then
$$D_g(r,d,\omega)=D_g(r,d-n_1(z),\omega').$$
\end{lem}

\begin{proof} One can also define the Hecke transformation of a family of parabolic sheaves (flat family yielding flat family, and preserve semistability). Thus, for $z\in I$, we have a morphism $$\mathrm{H}_z: \sU_{C,\,\omega}=\sU_C(r, d,\omega)\rightarrow \sU_C(r, d-n_1(z), \omega')=\sU_{C,\,\omega'}$$ such that $\mathrm{H}_z^{\ast}\Theta_{\sU_{C,\,\omega'}}=\Theta_{\sU_{C,\,\omega}}$.
In fact, $\mathrm{H}_z$ is an isomorphism. For any parabolic bundle $E'$ with quasi-parabolic structure of type $\vec n'(z)$, let
$$F_{i}(E')_z=ker\{E'_z\twoheadrightarrow Q_i(E')_z\}\quad (1\le i\le l_z).$$
Then there exists a bundle $E$ and a homomorphism $E'\xrightarrow{i} E$ such that $F_{l_z}(E')_z=ker\{E'_z\xrightarrow{i_z}E\}.$
Let $F_1(E)_z=i_z(E'_z)\subset E_z$ and
$$F_{i+1}(E)_z=i_z(F_i(E')_z).$$
Then the quasi-parabolic structure of $E$ at $z\in I$ given by
$$0=F_{l_z+1}(E)_z\subset F_{l_z}(E)_z\subset F_{l_z-1}(E)_z \subset \cdots\subset F_1(E)_z\subset E_z$$
has of type $\vec n(z)=(n_1(z),...,n_{l_z+1}(z))$ and the weights $\vec a(z)$ are determined by
$\vec a'(z)$ (let $a_1(z)=0$, $a_2(z)=k-a'_{l_z+1}(z)$ and $a_{i+1}(z)=a'_i(z)+k-a'_{l_z+1}(z)$ for $2\le i\le l_z$).
The construction can be applied to a family of parabolic sheaves, which induces $H_z^{-1}$.
\end{proof}

\begin{lem}\label{lem3.10} For $\omega=(k,\{\vec n(x),\vec a(x)\}_{x\in I})$, if $n_1(z)>1$, let
\ga{3.7} {\omega''=(k, \{\vec n(x), \vec a(x)\}_{x\neq z \in I}\cup \{\vec a''(z), \vec n''(z)\})}
where $\vec a''(z)=(0,a_2(z),\cdots, a_{l_z+1}(z), k)$ (we assume $a_1(z)=0$) and
$$\vec n''(z)=(n_1(z)-m,n_2(z),\cdots, n_{l_z+1}(z), m), \quad 1<m<n_1(z).$$
Then $D_g(r, d-n_1(z),\omega')=D_g(r,d-m,\omega'').$
\end{lem}

\begin{proof}
 For a semistable parabolic bundle $E$ with parabolic structures determined by $\omega''$, let its quasi-parabolic structure at $z\in I$ is given by
 $$E_z\twoheadrightarrow Q_{l_z+1}(E)_z \twoheadrightarrow Q_{l_z}(E)_z\twoheadrightarrow \cdots \twoheadrightarrow Q_1(E)_z\twoheadrightarrow Q_0(E)_z=0.$$
Let $E'=ker\{E\twoheadrightarrow Q_1(E)_z\}$, then $E'$ has quasi-parabolic structure
$$E'_z \twoheadrightarrow Q'_{l_z}(E)_z\twoheadrightarrow \cdots \twoheadrightarrow Q'_1(E)_z\twoheadrightarrow Q'_0(E)_z=0$$
of type $\vec n'(z)=(n_2(z),...,n_{l_z+1}(z),n_1(z))$ at $z\in I$, where
$$Q'_i(E)_z=ker\{Q_{i+1}(E)_z\twoheadrightarrow Q_1(E)_z\}, \quad (1\leq i\leq l_z).$$
Then we show that $E'$ is a semistable parabolic bundle with parabolic structure determined by $\omega'$ if and only if
$E$ is a semistable parabolic bundle with parabolic structure determined by $\omega''$. In fact, by direct computation, we have
$$\chi(E)+\frac{1}{k}\sum^{l_z+2}_{i=1}a_i''(z)n_i''(z)=\frac{r}{k}a_2(z)+\chi(E')+\frac{1}{k}\sum^{l_z+1}_{i=1}a_i'(z)n_i'(z),$$
which implies that ${\rm p}ar_{\omega''}\mu(E)=\frac{a_2(z)}{k}+{\rm p}ar_{\omega'}\mu(E')$. For any sub-bundle $\sF\subset E$, let $\sF'\subset E'$ be the sub-bundle
such that $$0\to \sF'\to\sF\to Q_1(E)^{\sF}_z\to 0$$ is an exact sequence of sheaves. Then ${\rm p}ar_{\omega''}\mu(\sF)={\rm p}ar_{\omega'}\mu(\sF')+\frac{a_2(z)}{k}$.
Thus $E'$ is semistable if and only if $E$ is semistable. The construction can be applied to a family of parabolic sheaves, which induces
$$\sU_{C,\,\omega''}=\mathcal{U}_C(r,d-m,\omega'')\xrightarrow{\varphi} \sU_{C,\,\omega'}=\mathcal{U}_C(r,d-n_1(z),\omega').$$
One check directly that $\varphi^*\Theta_{\sU_{C,\,\omega'}}=\Theta_{\sU_{C,\,\omega''}}$ (i.e., it pulls back an ample line bundle to an
ample line bundle), which implies that $\varphi$ is a finite surjective morphism. To show that $\varphi$ is a injective morphism, which implies that
$\varphi$ is an isomorphism since $\sU_{C,\,\omega'}$ and $\sU_{C,\,\omega'}$ are normal projective varieties, we note
$Q'_i(E)_z=Q_{i+1}(E)^{E'}_z\subset Q_{i+1}(E)_z$ is the image of $E'_z\to E_z\twoheadrightarrow Q_{i+1}(E)_z$. Then $(E',Q'_{\bullet}(E)_z)=(E',Q_{\bullet+1}(E)^{E'}_z)$
is a parabolic subsheaf of $(E,Q_{\bullet+1}(E)_z)$ and we have exact sequence
$$0\to (E',Q_{\bullet+1}(E)^{E'}_z) \to (E,Q_{\bullet+1}(E)_z)\to ( \,_zQ_1(E)_z, Q_1(E)_{\bullet+1})\to 0$$
of parabolic sheaves, where
$$Q_1(E)_{\bullet+1}: Q_1(E)_z\twoheadrightarrow Q_1(E)_z\twoheadrightarrow\cdots\twoheadrightarrow Q_1(E)_z\twoheadrightarrow 0.$$
By direct computations, we have
$${\rm p}ar_{\omega''}\mu((E',Q_{\bullet+1}(E)^{E'}_z))={\rm p}ar_{\omega''}\mu((E,Q_{\bullet+1}(E)_z)).$$
Thus $(E,Q_{\bullet+1}(E)_z)$ is $s$-equivalent to $$(E',Q_{\bullet+1}(E)^{E'}_z)\oplus  ( \,_zQ_1(E)_z, Q_1(E)_{\bullet+1}),$$
which implies that $\varphi$ is a injective morphism, and we are done.
\end{proof}

\begin{rmks}\label{rmks3.11} (1) The moduli spaces $\sU_{C,\,\omega''}$ and theta line bundles $\Theta_{\sU_{C,\,\omega''}}$ are constructed
in \cite{Su3} for the case $a_{l_z+1}(z)-a_1(z)=k$, vanishing theorems can be generalized to this case.

(2) Let $\omega''=H_z^m(\omega)$, we will simply call $H_z^m(\omega)$ a Hecke transformation of $\omega$ at $z\in I$. Then
\ga{3.8} { D_g(r,d,\omega)=D_g(r, d-m, H^m_z(\omega)).}
\end{rmks}

Now we can prove another version of recurrence relation \eqref{3.4}, in which the degree $d$ is kept unchanged.

\begin{thm}\label{thm3.12} For any partitions $g=g_1+g_2$ and $I=I_1\cup I_2$, let
$$W_k=\{\,\lambda=(\lambda_1,...,\lambda_r)\,|\, 0=\lambda_r\le\lambda_{r-1}\le\cdots\le\lambda_1\le k\,\}$$
$$W'_k=\left\{\,\lambda\in W_k\,\,\mid\,\,\left(\sum_{x\in
I_1}\sum^{l_x}_{i=1}d_i(x)r_i(x)+\sum^r_{i=1}\lambda_i\right) \equiv 0({\rm mod}\,\,r)\right\}.$$
Then we have the following recurrence relation
\ga{3.9} {D_g(r,d,\omega)=\sum_{\mu\in W'_k} D_{g_1}(r, 0,\omega_1^{\mu})\cdot D_{g_2}(r,d,\omega_2^{\mu}).}
\end{thm}

\begin{proof} Let $P_k=\{\mu=(\mu_1,...,\mu_r)\,|\, 0\le\mu_r\le\cdots\le\mu_1< k\,\}$, by the recurrence relation \eqref{3.4}, we have
$$D_g(r,d,\omega)=\sum_{\mu\in Q_k}D_{g_1}(r,d_1^{\mu},\omega_1^{\mu})\cdot D_{g_2}(r,d_2^{\mu},\omega_2^{\mu})$$
where $Q_k=\{\mu=(\mu_1,\cdots, \mu_r)\in P_k\,|\, d_1^{\mu}\in \mathbb{Z}\}$. Recall definition of $d^{\mu}_j$ in
Theorem \ref{thm3.6}, which are integers such that $d_1^{\mu}+d_2^{\mu}=d$ and
\ga{3.10} {k(d_1^{\mu}+r)=k\cdot n_1^{\omega}+|\mu|,\quad |\mu|=\sum^r_{i=1}\mu_i}
where $n_1^{\omega}$ is the rational number defined in Theorem \ref{thm3.6}.

For $\mu=(\mu_1,\cdots, \mu_r),$ $0\leq \mu_r\leq \cdots \leq \mu_1\leq k$, let
$$H^1(\mu)=(k-\mu_{r-1}+\mu_r,\mu_1-\mu_{r-1},\mu_2-\mu_{r-1},\cdots, \mu_{r-2}-\mu_{r-1},0),$$
$H^{m}(\mu):=H^1(H^{m-1}(\mu))$ for $2\leq m\leq r$. Then, when $1\le m<r$,
$$H^m(\mu)_j=\left\{
\begin{array}{llll} k-\mu_{r-m}+\mu_{r-m+j} &\mbox{when $1\le j\le m$}\\
\mu_{j-m}-\mu_{r-m}&\mbox{when
$j>m $}\end{array}\right.$$
and $H^r(\mu)=(\mu_1-\mu_r,\mu_2-\mu_r,...,\mu_{r-1}-\mu_r,0)$. Moreover
\ga{3.11} { |H^m(\mu)|=\left\{
\begin{array}{llll} k\cdot m-r\cdot \mu_{r-m}+|\mu| &\mbox{when $m<r$}\\
-r\cdot\mu_r+|\mu|&\mbox{when
$m=r $}\end{array}\right. }
Let $0\le i^{\mu}<r$ be the unique integer such that $d_1^{\mu}\equiv i^{\mu}({\rm mod}\,\,r)$, let
$$\phi(\mu):=H^{r-i^{\mu}}(\mu).$$
Then, by \eqref{3.11}, it is easy to see that we have a map
\ga{3.12} {\phi:Q_k\to W'_k.}
One can check that $\omega_i^{\phi(\mu)}$ is a Hecke transformation of $\omega_i^{\mu}$ ($i=1,\,2$),
thus $D_{g_1}(r,d^{\mu}_1,\omega_1^{\mu})=D_{g_1}(r,0,\omega_1^{\phi(\mu)})$, $D_{g_2}(r,d_2^{\mu},\omega_2^{\mu})=D_{g_2}(r,d,\omega_2^{\phi(\mu)})$
by Lemma \ref{lem3.9} and Lemma \ref{lem3.10}. To prove the recurrence relation \eqref{3.9}, it is enough to show that
$\phi$ is bijective.

To prove the injectivity of $\phi$, let $\phi(\mu)=\phi(\mu')$, it is enough to show $i^{\mu}=i^{\mu'}$.
If both $i^{\mu}$ and $i^{\mu'}$ are nonzero, note $|\phi(\mu)|=k(r-i^{\mu})-r\mu_{i^{\mu}}+|\mu|$,
by $\phi(\mu)=\phi(\mu')$ and \eqref{3.10}, there exists a $q\in\mathbb{Z}$ such that
$$r\cdot(\mu'_{i^{\mu'}}-\mu_{i^{\mu}})=k(d_1^{\mu'}-i^{\mu'}-(d_1^{\mu}-i^{\mu}))=k\cdot r\cdot q.$$
Thus $k>|\mu'_{i^{\mu'}}-\mu_{i^{\mu}}|=k|q|$, which implies $q=0$ and $\mu'_{i^{\mu'}}=\mu_{i^{\mu}}$. If $i^{\mu}\neq i^{\mu'}$, let
$a=i^{\mu}-i^{\mu'}>0$,  then formula
\ga{3.13} {\phi(\mu)_j=\left\{
\begin{array}{llll} k-\mu_{i^{\mu}}+\mu_{j+i^{\mu}} &\mbox{when $1\le j\le  r-i^{\mu}$}\\
\mu_{j-r+i^{\mu}}-\mu_{i^{\mu}}&\mbox{when
$j>r-i^{\mu} $}\end{array}\right.}
implies $\mu_a=k+\mu_r'\ge k$, which is a contradiction since $\mu\in Q_k$. If
$i^{\mu}=0$, $i^{\mu'}$ must be zero. Otherwise, the same arguments imply $\mu'_{i^{\mu'}}=\mu_r$ and $\mu_j=k+\mu'_{i^{\mu'}+j}$
for all $1\le j\le r-i^{\mu'}$.

To prove that $\phi$ is surjective, by using \eqref{3.11}, \eqref{3.10} becomes
\ga{3.14} {\frac{k\cdot n_1^{\omega}+|\phi(\mu)|}{r}=\left\{
\begin{array}{llll} k\cdot\frac{d_1^{\mu}+2r-i^{\mu}}{r}-\mu_{i^{\mu}} &\mbox{when $i^{\mu}>0$}\\
k\cdot\frac{d_1^{\mu}+r}{r}-\mu_r &\mbox{when $i^{\mu}=0 $}\end{array}\right.}
For any $\lambda=(\lambda_1,...,\lambda_{r-1},0)\in W'_k$, there are unique integers $q^{\lambda}$ and $0\le r^{\lambda}<k$ such that
$$\frac{k\cdot n_1^{\omega}+|\lambda|}{r}=k\cdot q^{\lambda}-r^{\lambda}.$$
If $\lambda_1+r^{\lambda}<k$, let $\mu=(\lambda_1+r^{\lambda}, ..., \lambda_{r-1}+r^{\lambda},r^{\lambda})\in P_k$, then $d_1^{\mu}=r(q^{\lambda}-1)$
by \eqref{3.10}. Thus $i^{\mu}=0$ and $\phi(\mu)=\lambda$. If $\lambda_1+r^{\lambda}\ge k$, since $\lambda_r+r^{\lambda}<k$, there exists an unique $1\le i_0\le r-1$ such that $$\lambda_{i_0}+r^{\lambda}\ge k, \qquad \lambda_{i_0+1}+r^{\lambda}< k.$$
Let $\mu_j=\lambda_{i_0+j}+r^{\lambda}$ ($1\le j\le r-i_0$) and $\mu_{r-i_0+j}=\lambda_j+r^{\lambda}-k$ ($1\le j\le i_0$).
Then $\mu=(\mu_1,...,\mu_r)\in Q_k$ with $d_1^{\mu}=r(q^{\lambda}-1)-i_0$ and $i^{\mu}=r-i_0$. It is easy to see that $\phi(\mu)=\lambda$.
\end{proof}

\section {Proof of Verlinde formula}

As an application of the recurrence relation \eqref{3.3} and \eqref{3.9}, we prove a closed formula of $D_g(r,d,\omega)$ (the so called Verlinde formula).
Recall
\ga{4.1} {S_{\lambda}(z_1,...,z_r)=\frac{|z_j^{\lambda_i+r-i}|}{|z_j^{r-i}|}=\frac{|z_j^{\lambda_i+r-i}|}{\Delta(z_1,...,z_r)}} is the so called
Schur polynomial of $\lambda=(\lambda_1\ge\lambda_2\ge\cdots\ge\lambda_r\ge 0)$,
$$\Delta(z_1,...,z_r)=\prod_{i<j}(z_i-z_j).$$
We give here an detail proof of some identities of Schur polynomials.

\begin{prop}\label{prop4.1} For $\vec v=(v_1,\ldots,v_r)$, $0\le v_r<\cdots <v_1<r+k$, let
$$S_{\lambda}\left({\rm exp}\,2\pi i\frac{\vec v}{r+k}\right)=S_{\lambda}(e^{2\pi i\frac{v_1}{r+k}},...,e^{2\pi i\frac{v_r}{r+k}}),$$
$P_k=\{\mu=(\mu_1,...,\mu_r)\,|\, 0\le\mu_r\le\cdots\le\mu_1< k\,\}$, $|\mu|:=\sum\mu_i$.
Then
\ga{4.2} {\aligned&\sum_{\mu\in P_k}S_{\mu}\left({\rm exp}\,2\pi i\frac{\vec v}{r+k}\right)\cdot S_{\mu^*}\left({\rm exp}\,2\pi i\frac{\vec v}{r+k}\right)\\&=
{\rm exp}\left(2\pi i\frac{k}{r+k}|\vec v|\right)\cdot\frac{k(r+k)^{r-1}}{\prod_{i<j}\left(2\sin\,\pi \frac{v_i-v_j}{r+k}\right)^2},\endaligned}
let $W_k=\{\mu=(\mu_1,...,\mu_r)\,|\,0=\mu_r\le\mu_{r-1}\le\cdots\le\mu_1\le k\,\}$, we have
\ga{4.3}{\aligned&\sum_{\mu\in W_k} S_{\mu}\left({\rm exp}\,2\pi i\frac{\vec v}{r+k}\right)\cdot S_{\mu^*}\left({\rm exp}\,2\pi i\frac{\vec {v}}{r+k}\right)\\&={\rm exp}\left(2\pi i\frac{k}{r+k}|\vec v|\right)\cdot\frac{r(r+k)^{r-1}}{\prod_{i<j}\left(2\sin\,\pi \frac{v_i-v_j}{r+k}\right)^2}\endaligned}
and, if $\vec v\nsim\vec{v'}$ (here $\vec v\sim\vec{v'}\,\Leftrightarrow\,\vec v-\vec{v'}=(a,...,a)$ for some $a\in\mathbb{Z}$),
\ga{4.4}{\aligned \sum_{\mu\in W_k}
&{\rm exp}\,2\pi i\frac{-|\mu|\cdot |\vec v|}{r(r+k)}\cdot
{\rm exp}\,2\pi i\frac{-|\mu^*|\cdot |\vec {v'}|}{r(r+k)}\cdot\\& S_{\mu}\left({\rm exp}\,2\pi i\frac{\vec v}{r+k}\right)\cdot S_{\mu^*}\left({\rm exp}\,2\pi i\frac{\vec {v'}}{r+k}\right)
=0.\endaligned}
\eqref{4.2} and \eqref{4.3} are invariant under the equivalence relation $\vec v\sim\vec{v'}$.
\end{prop}

\begin{proof} To prove \eqref{4.2}, since $\mathbb{S}_{\mu^*}(V)={\rm det}(V)^k\otimes \mathbb{S}_{\mu}(V^*)$, we have
$$S_{\mu^*}\left({\rm exp}\,2\pi i\frac{\vec v}{r+k}\right)=\overline{S_{\mu}\left({\rm exp}\,2\pi i\frac{\vec v}{r+k}\right)}{\rm exp}\left(2\pi i\frac{k}{r+k}\sum_{i=1}^{r}v_i\right).$$ Thus it is enough to show that
\ga{4.5}{\aligned&\sum_{\mu\in P_k}S_{\mu}\left({\rm exp}\,2\pi i\frac{\vec v}{r+k}\right)\cdot \overline{S_{\mu}\left({\rm exp}\,2\pi i\frac{\vec v}{r+k}\right)}\\&=
\frac{k(r+k)^{r-1}}{\prod_{i<j}\left(2\sin\,\pi \frac{v_i-v_j}{r+k}\right)^2}.\endaligned}

For $\lambda=(\lambda_1,...,\lambda_r)$, the functions $e^{\tau(\vec{\lambda})}$, $J(e^{\vec\lambda})$ are defined by
$$\aligned &e^{\tau(\vec\lambda)}(\mathrm{diag}(t_1, \cdots,t_r)):=t_1^{\lambda_{\tau(1)}}\cdot \cdots\cdot t_r^{\lambda_{\tau(r)}}\\& J(e^{\vec\lambda})(\mathrm{diag}(t_1, \cdots,t_r)):=\sum_{\tau\in \mathfrak{S}_r}\epsilon(\tau)e^{\tau(\vec\lambda)}(\mathrm{diag}(t_1, \cdots,t_r)),\endaligned$$
where $\tau(\vec\lambda)=(\lambda_{\tau(1)},...,\lambda_{\tau(r)})$, $\mathfrak{S}_r$ is the symmetric group. Let
$$\Delta(\vec v)=\prod_{i<j}(e^{2\pi i\frac{v_i}{r+k}}-e^{2\pi i\frac{v_j}{r+k}})$$
and $\rho=(r-1, r-2, ..., 0)$. By expansion of determinant, we have
$$\aligned &S_{\mu}\left({\rm exp}\,2\pi i\frac{\vec v}{r+k}\right)=\frac{1}{\Delta(\vec v)}\sum_{\tau\in \mathfrak{S}_r}\epsilon(\tau)e^{2\pi i\frac{\mu_1+r-1}{r+k}v_{\tau(1)}}
\cdot\cdots\cdot e^{2\pi i\frac{\mu_r}{r+k}v_{\tau(r)}}\\=&\frac{1}{\Delta(\vec v)}\sum_{\tau\in \mathfrak{S}_r}\epsilon(\tau)e^{\tau(\vec v)}\left({\rm exp}\,2\pi i\frac{\vec\mu+\vec\rho}{r+k}\right)=\frac{1}{\Delta(\vec v)}J(e^{\vec v})(t_{\mu}),\endaligned$$
where $t_{\mu}={\rm exp}\,2\pi i\frac{\vec\mu+\vec\rho }{r+k}$.
By $\Delta(\vec v)\overline{\Delta(\vec v)}=\prod_{i<j}\left(2\sin\,\pi \frac{v_i-v_j}{r+k}\right)^2$, we have
\ga{4.6} {\aligned &S_{\mu}\left({\rm exp}\,2\pi i\frac{\vec v}{r+k}\right)\cdot \overline{S_{\mu}\left({\rm exp}\,2\pi i\frac{\vec v}{r+k}\right)}
=\\&\frac{1}{\prod_{i<j}\left(2\sin\,\pi \frac{v_i-v_j}{r+k}\right)^2}J(e^{\vec v)})(t_{\mu})
\cdot \overline{J(e^{\vec v})(t_{\mu})}.\endaligned}

Let $T_k=\{\,t=\mathrm{diag}(e^{\frac{2\pi i}{r+k}t_1}, \cdots, e^{\frac{2\pi i}{r+k}t_r})\,|\, 0\le t_i<r+k\,\}\subset {\rm GL}(r)$ be the subgroup and
$T_k^{reg}=\{\,t\in T_k\,|\, t_i\ne t_j \text{if $i\neq j$}\,\}$. The group $\mathfrak{S}_r$ acts on $T_k$ by
$\tau(t)=\mathrm{diag}(e^{\frac{2\pi i}{r+k}t_{\tau(1)}}, \cdots, e^{\frac{2\pi i}{r+k}t_{\tau(r)}})$ and the functions
$$\aligned J(e^{\vec\lambda})(t)&=\sum_{\tau\in \mathfrak{S}_r}\epsilon(\tau)e^{2\pi i\frac{\lambda_{\tau(1)}}{r+k}t_{1}}
\cdot\cdots\cdot e^{2\pi i\frac{\lambda_{\tau(r)}}{r+k}t_{r}}\\&=
\sum_{\tau\in \mathfrak{S}_r}\epsilon(\tau)e^{2\pi i\frac{\lambda_1}{r+k}t_{\tau(1)}}
\cdot\cdots\cdot e^{2\pi i\frac{\lambda_r}{r+k}t_{\tau(r)}}\endaligned$$
for any $\lambda=(\lambda_1,...,\lambda_r)$ are ant-symmetric functions, thus $J(e^{\vec\lambda})(t)=0$ if $t\notin T_k^{reg}$. It is clear that $\mathfrak{S}_r$ acts
on $T_k^{reg}$ freely and
$$T^{reg}_k=\bigcup_{\mu\in \bar P_k}\mathfrak{S}_r\cdot t_{\mu},\quad t_{\mu}={\rm exp}\,2\pi i\frac{\mu+\rho }{r+k}.$$
The right hand side of \eqref{4.6} is a symmetric function, we have
$$\aligned &\sum_{\mu\in \bar P_k}S_{\mu}\left({\rm exp}\,2\pi i\frac{\vec v}{r+k}\right)\cdot \overline{S_{\mu}\left({\rm exp}\,2\pi i\frac{\vec v}{r+k}\right)}=
\frac{1}{|\mathfrak{S}_r|}\\&
\prod_{i<j}\left(2\sin\,\pi \frac{v_i-v_j}{r+k}\right)^{-2}\sum_{t\in T_k}J(e^{\vec v})(t)
\overline{J(e^{\vec v})(t)}\endaligned$$ where $\bar P_k=\{\mu=(\mu_1,...,\mu_r)\,|\,0\le\mu_r\le\cdots\le\mu_1\le k\}$.
To compute
$$\aligned&\sum_{t\in T_k}J(e^{\vec v})(t)
\overline{J(e^{\vec v})(t)}=\sum_{\tau,\sigma\in\mathfrak{S}_r}\epsilon(\tau)\cdot \epsilon(\sigma)\sum_{t\in T_k}
e^{\tau(\vec v)}(t)\cdot\overline{e^{\sigma(\vec v)}(t)},\endaligned$$
note $e^{\tau(\vec v)}$ and $e^{\sigma(\vec v)}$ are different character of $T_k$ when $\tau\neq \sigma$, we have
$$\sum_{t\in T_k}J(e^{\vec v})(t)\overline{J(e^{\vec v})(t)}=|\mathfrak{S}_r|\cdot |T_k|.$$  Thus
\ga{4.7}{\aligned&\sum_{\mu\in \bar P_k}S_{\mu}\left({\rm exp}\,2\pi i\frac{\vec v}{r+k}\right)\cdot \overline{S_{\mu}\left({\rm exp}\,2\pi i\frac{\vec v}{r+k}\right)}
\\&=\frac{(r+k)^r}{\prod_{i<j}\left(2\sin\,\pi \frac{v_i-v_j}{r+k}\right)^2}.\endaligned}

For $\mu\in P'_k:=\bar P_k\setminus P_k$, let
$t_{\mu}'=\mathrm{diag}(1, e^{2\pi i\frac{\mu_2+r-1}{r+k}}, \cdots, e^{2\pi i\frac{\mu_r+1}{r+k}} )$ and
$$T'_k=\{\,t=\mathrm{diag}(1, e^{\frac{2\pi i}{r+k}t_2}, \cdots, e^{\frac{2\pi i}{r+k}t_r})\,|\, 0\le t_i<r+k\,\}\subset T_k$$ be the subgroup and
${T'}_k^{reg}=T'_k\cap T^{reg}_k$. Then
$${T'}^{reg}_k=\bigcup_{\mu\in  P'_k}\mathfrak{S}_{r-1}\cdot t'_{\mu}.$$
Note $J(e^{\vec v})(t_{\mu})=e^{-2\pi i\frac{|\vec v|}{r+k}}J(e^{\vec v})(t'_{\mu})$, $J(e^{\vec v})(t)=0$ if $t\notin {T'}_k^{reg}$, we have
\ga{4.8} {\aligned &\sum_{\mu\in P'_k}S_{\mu}\left({\rm exp}\,2\pi i\frac{\vec v}{r+k}\right)\cdot \overline{S_{\mu}\left({\rm exp}\,2\pi i\frac{\vec v}{r+k}\right)}=\\&
\frac{1}{|\mathfrak{S}_{r-1}|}
\prod_{i<j}\left(2\sin\,\pi \frac{v_i-v_j}{r+k}\right)^{-2}\sum_{t\in  T'_k}J(e^{\vec v})(t)
\overline{J(e^{\vec v})(t)}\\&=\frac{r(r+k)^{r-1}}{\prod_{i<j}\left(2\sin\,\pi \frac{v_i-v_j}{r+k}\right)^2}.\endaligned}
Thus \eqref{4.7} and \eqref{4.8} imply the formula \eqref{4.2}. The proof of formula \eqref{4.3} is similar with formula \eqref{4.8}, we omit it.

Now we are going to prove formula \eqref{4.4}. To simpify notation, let
$$G_{\mu}(\vec v):={\rm exp}\,2\pi i\frac{-|\mu|\cdot |\vec {v}|}{r(r+k)}\cdot S_{\mu}\left({\rm exp}\,2\pi i\frac{\vec v}{r+k}\right).$$
Then it is equivalent to prove that, when $\vec v\nsim\vec{v'}$, we have
\ga{4.9}{\sum_{\mu\in W_k}G_{\mu}(\vec v)\overline{G_{\mu}(\vec {v'})}=0.}
Let $\lambda^{\mu}=(\lambda_1^{\mu},...,\lambda^{\mu}_r)$ with $\lambda_i^{\mu}=\mu_i+r-i-\frac{|\mu|+|\rho|}{r}$, then
$$\aligned G_{\mu}(\vec v)&=\frac{{\rm exp}\,2\pi i\frac{|\rho|\cdot |\vec {v}|}{r(r+k)}}{\Delta(\vec v)}\sum_{\tau\in \mathfrak{S}_r}\epsilon(\tau)e^{2\pi i\frac{\lambda^{\mu}_1}{r+k}v_{\tau(1)}}
\cdot\cdots\cdot e^{2\pi i\frac{\lambda^{\mu}_r}{r+k}v_{\tau(r)}}\\&=\frac{{\rm exp}\,2\pi i\frac{|\rho|\cdot |\vec {v}|}{r(r+k)}}{\Delta(\vec v)}\sum_{\tau\in \mathfrak{S}_r}\epsilon(\tau)e^{\tau(\vec v)}\left({\rm exp}\,2\pi i\frac{\lambda^{\mu}}{r+k}\right)\\&=\frac{{\rm exp}\,2\pi i\frac{|\rho|\cdot |\vec {v}|}{r(r+k)}}{\Delta(\vec v)}J(e^{\vec v})(t_{\lambda^{\mu}}),\quad t_{\lambda^{\mu}}={\rm exp}\,2\pi i\frac{\lambda^{\mu}}{r+k}.\endaligned$$
Since $e^{\sigma(\vec v)}$, $e^{\tau(\vec {v'})}$ ($\forall\,\sigma,\,\tau\in\mathfrak{S}_r$) are different
characters of a subgroup $T_k=\{\,t=\mathrm{diag}(e^{\frac{2\pi i}{r+k}t_1}, \cdots, e^{\frac{2\pi i}{r+k}t_r})\,|\sum t_i=0,\,t_i-t_j\in\mathbb{Z}\,\}\subset {\rm GL}(r)$
whenever $\vec v\nsim\vec {v'}$, we have
$$\aligned \sum_{\mu\in W_k}G_{\mu}(\vec v)\overline{G_{\mu}(\vec {v'})}&=
\frac{{\rm exp}\,2\pi i\frac{|\rho|\cdot (|\vec {v}|-|\vec{v'}|)}{r(r+k)}}{\Delta(\vec v)\overline{\Delta(\vec v')}}
\sum_{\mu\in W_k}J(e^{\vec v})(t_{\lambda^{\mu}})\cdot\overline{J(e^{\vec {v'}})(t_{\lambda^{\mu}})}\\&=
\frac{{\rm exp}\,2\pi i\frac{|\rho|\cdot (|\vec {v}|-|\vec{v'}|)}{r(r+k)}}{\Delta(\vec v)\overline{\Delta(\vec v')}|\mathfrak{S}_r|}
\sum_{t\in T_k}J(e^{\vec v})(t)\cdot\overline{J(e^{\vec {v'}})(t)}=0.
\endaligned$$
\end{proof}

\begin{nota}\label{nota4.2} For $\vec n(x)=(n_1(x),n_2(x),\cdots,n_{l_x+1}(x))$ and $$\vec a(x)=(a_1(x),a_2(x),\cdots,a_{l_x+1}(x))$$ with
$\sum n_i(x)=r$, $0\leq a_1(x)<a_2(x)<\cdots
<a_{l_x+1}(x)< k$, define
\ga{4.10} {\lambda_x=(\,\,\overbrace{\lambda_1,\ldots,\lambda_1}^{n_1(x)}\,,\,\,\overbrace{\lambda_2,\ldots,\lambda_2}^{n_2(x)}\,,\,\,\ldots,\, \,\overbrace{\lambda_{l_x+1},\ldots,\lambda_{l_x+1}}^{n_{l_x+1}(x)}\,\,)}
where $\lambda_i=k-a_i(x)$ ($1\le i\le l_x+1$).
\end{nota}

\begin{thm}\label{thm4.3} For given data $\omega=(k,\{\vec n(x),\vec a(x)\}_{x\in I})$, let
$$S_{\omega}(z_1,...,z_r)=\prod_{x\in I}S_{\lambda_x}(z_1,...,z_r),\quad |\omega|=\sum_{x\in I}|\lambda_x|$$
where $S_{\lambda_x}(z_1,...,z_r)$ are Schur polynomials and $|\lambda_x|$ denotes the total number of boxes in a Young diagram
associated to $\lambda_x$. Then
\ga{4.11}{\aligned &D_g(r,d,\omega)=(-1)^{d(r-1)}\left(\frac{k}{r}\right)^g(r(r+k)^{r-1})^{g-1}\\&\sum_{\vec v}
\frac{{\rm exp}\left(2\pi i\left(\frac{d}{r}-\frac{|\omega|}{r(r+k)}\right)\sum_{i=1}^{r}v_i\right)S_{\omega}\left({\rm exp}\,2\pi i\frac{\vec v}{r+k}\right)} {\prod_{i<j}\left(2\sin\,\pi \frac{v_i-v_j}{r+k}\right)^{2(g-1)}}\endaligned}
where $\vec v=(v_1,v_2,\ldots,v_r)$ runs through the integers $$0=v_r<\cdots <v_2<v_1<r+k.$$
\end{thm}

\begin{proof} Let $V_g(r,d,\omega)$ denote the right hand side of
formula \eqref{4.11} (the Verlinde number). When $|I|=0$, we define $V_g(r,d,\omega)$ to be
$$(-1)^{d(r-1)}\left(\frac{k}{r}\right)^g(r(r+k)^{r-1})^{g-1}\sum_{\vec v}
\frac{{\rm exp}\left(2\pi i\frac{d}{r}\sum_{i=1}^{r}v_i\right)} {\prod_{i<j}\left(2\sin\,\pi \frac{v_i-v_j}{r+k}\right)^{2(g-1)}}.$$
Note that both $V_g(r,d,\omega)$ and $D_g(r,d,\omega)$ (even the moduli space $\sU_{C,\,\omega}$ and theta line bundle $\Theta_{\sU_{C,\,\omega}}$ on it)
are invariant under the equivalence relation: $\lambda_x\sim\lambda'_x\,\Leftrightarrow\,\lambda_x-\lambda_x'=(a_x,a_x,...,a_x)$ for some integer $a_x\in \mathbb{Z}$.
Assume that $D_g(r,d,\omega)=V_g(r,d,\omega)$ when $|I|\le 3$ (we will prove it later). Then the proof is done by the following lemmas.
\end{proof}

\begin{lem}\label{lem4.4} If the formula \eqref{4.11} holds when $g=0$, then it holds for any $g>0$.
\end{lem}

\begin{proof} Recall the recurrence relation \eqref{3.3}, we have
\ga{4.12} {D_g(r,d,\omega)=\sum_{\mu}D_{g-1}(r,d,\omega^{\mu})}
where $\omega^{\mu}=(k, \{\vec n(x),\,\vec a(x)\}_{x\in I\cup\{x_1,\,x_2\}})$ was defined in Notation \ref{nota3.5} and
$\mu=(\mu_1,\ldots,\mu_r)$ runs through the integers $0\le\mu_r\le\cdots\le\mu_1< k$.

It is easy to check that $\lambda_{x_2}=(k-\mu_r,\ldots,k-\mu_1)$ and
$$\lambda_{x_1}=(\mu_1,\ldots,\mu_r)+(\mu_1+\mu_r-k,\,\mu_1+\mu_r-k,\ldots,\, \mu_1+\mu_r-k)$$
(in Notation \ref{nota4.2}). Thus, without loss of generality, we can assume
$$\lambda_{x_1}=\mu=(\mu_1,\ldots,\mu_r),\quad \lambda_{x_2}=\mu^*=(k-\mu_r,\ldots,k-\mu_1).$$
Assume that formula \eqref{4.11} holds for $g-1$, then
\ga{4.13}{\aligned &D_{g-1}(r,d,\omega^{\mu})=(-1)^{d(r-1)}\left(\frac{k}{r}\right)^{g-1}(r(r+k)^{r-1})^{g-2}\\&\sum_{\vec v}
\frac{{\rm exp}\left(2\pi i\left(\frac{d}{r}-\frac{|\omega^{\mu}|}{r(r+k)}\right)\sum_{i=1}^{r}v_i\right)S_{\omega^{\mu}}\left({\rm exp}\,2\pi i\frac{\vec v}{r+k}\right)} {\prod_{i<j}\left(2\sin\,\pi \frac{v_i-v_j}{r+k}\right)^{2(g-2)}}\endaligned}
where $|\omega^{\mu}|=|\omega|+k\cdot r$, $S_{\omega^{\mu}}
=S_{\omega}\cdot S_{\mu}\cdot S_{\mu^*}$. By \eqref{4.12} and \eqref{4.13},
$$\aligned&D_g(r,d,\omega)=(-1)^{d(r-1)}\left(\frac{k}{r}\right)^g(r(r+k)^{r-1})^{g-1}\\&\sum_{\vec v}
\frac{{\rm exp}\left(2\pi i\left(\frac{d}{r}-\frac{|\omega|}{r(r+k)}\right)\sum_{i=1}^{r}v_i\right)S_{\omega}\left({\rm exp}\,2\pi i\frac{\vec v}{r+k}\right)} {\prod_{i<j}\left(2\sin\,\pi \frac{v_i-v_j}{r+k}\right)^{2(g-1)}}\\&{\rm exp}\left(-2\pi i\frac{k}{r+k}\sum_{i=1}^{r}v_i\right)
\frac{\prod_{i<j}\left(2\sin\,\pi \frac{v_i-v_j}{r+k}\right)^2}{k(r+k)^{r-1}}\\&
\sum_{\mu}S_{\mu}\left({\rm exp}\,2\pi i\frac{\vec v}{r+k}\right)\cdot S_{\mu^*}\left({\rm exp}\,2\pi i\frac{\vec v}{r+k}\right).\endaligned$$
Then the formula \eqref{4.11} holds by the identity \eqref{4.2} in Proposition \ref{prop4.1}.

\end{proof}

\begin{lem}\label{lem4.5} If the formula \eqref{4.11} for $D_0(r,d,\omega)$ holds when $|I|\le3$, then it holds for all $D_0(r,d,\omega)$.
\end{lem}

\begin{proof} The proof is by induction on the number of parabolic points. By Theorem \ref{thm3.12}, let $I=I_1\cup I_2$ with $|I_1|=2$, we have
$$D_0(r,d,\omega)=\sum_{\mu\in W'_k}V_0(r,0,\omega_1^{\mu})\cdot V_0(r,d,\omega_2^{\mu}).$$
It is not difficult to check that $V_0(r,0,\omega_1^{\mu})=0$ for $\mu\in W_k\setminus W'_k$. Thus
$$\aligned &D_0(r,d,\omega)=\sum_{\mu\in W_k}V_0(r,0,\omega_1^{\mu})\cdot V_0(r,d,\omega_2^{\mu})=\frac{(-1)^{d(r-1)}}{(r(r+k)^{r-1})^2}\\&\sum_{\vec v,\,\vec {v'}}
\frac{{\rm exp}\left(2\pi i\left(-\frac{|\omega_1|}{r(r+k)}\right)|\vec v|\right)} {\prod_{i<j}\left(2\sin\,\pi \frac{v_i-v_j}{r+k}\right)^{-2}}
\cdot\frac{{\rm exp}\left(2\pi i\left(\frac{d}{r}-\frac{|\omega_2|}{r(r+k)}\right)|\vec {v'}|\right)} {\prod_{i<j}\left(2\sin\,\pi \frac{v'_i-v'_j}{r+k}\right)^{-2}}\cdot\\&
S_{\omega_1}\left({\rm exp}\,2\pi i\frac{\vec v}{r+k}\right)\cdot S_{\omega_2}\left({\rm exp}\,2\pi i\frac{\vec {v'}}{r+k}\right)\cdot\sum_{\mu\in W_k}
{\rm exp}\,2\pi i\frac{-|\mu|\cdot |\vec v|}{r(r+k)}\\&
{\rm exp}\,2\pi i\frac{-|\mu^*|\cdot |\vec {v'}|}{r(r+k)}\cdot S_{\mu}\left({\rm exp}\,2\pi i\frac{\vec v}{r+k}\right)\cdot S_{\mu^*}\left({\rm exp}\,2\pi i\frac{\vec {v'}}{r+k}\right)
\endaligned$$ and we are done by \eqref{4.3} and \eqref{4.4} of Proposition \ref{prop4.1}. Here we remark that $\vec v\nsim\vec{v'}$ if and only if $\vec v\neq\vec{v'}$
since our $\vec v$, $\vec{v'}$ satisfy $v_r=v'_r=0$.
\end{proof}

Finally, we prove $D_0(r,d,\omega)=V_0(r,d,\omega)$ when $|I|\le 3$. The
data $$\omega=(k,\{\vec n(x),\vec a(x)\}_{x\in I})$$
is encoded in partitions $\omega=\{\lambda_x\}_{x\in I}$ (See Notation \ref{nota4.2}) where
$$\lambda_x=(\lambda_1(x),\lambda_2(x),...,\lambda_r(x)),\quad k\ge\lambda_1(x)\ge\lambda_2(x)\ge\cdots\ge \lambda_r(x)\ge 0.$$
Thus, for convenience of computations, we will use notations
$$D_g(r,d,\{\lambda_x\}_{x\in I}):=D_g(r,d,\omega), \quad V_g(r,d,\{\lambda_x\}_{x\in I}):=V_g(r,d,\omega),$$
$$\omega_s:=(\,\overbrace{1,\ldots,1}^{s}\,,\,\overbrace{0,\ldots,0}^{r-s}\,)\quad (1\le s\le r).$$
Let $V$ be standard representation of ${\rm GL}_r(\mathbb{C})$ and $1\le s\le r-1$, then
\ga{4.14} {\mathbb{S}_{\lambda}(V)\otimes \mathbb{S}_{\omega_s}(V)=\bigoplus_{\mu\in\Y(\lambda,\,\omega_s)}\mathbb{S}_{\mu}(V)}
where the Young diagrams of partitions $\mu\in Y(\lambda,\,\omega_s)$ are obtained from $\lambda$ by adding $s$ boxes with no two in the same row
(See (6.9) at page 79 of \cite{FH}). \textbf{In the rest of the article, without loss of generality, we assume $\lambda_r(x)=0$ for all partitions $\lambda_x$ ($x\in I$)}.
We compute $V_0(r,0,\{\lambda_x\}_{x\in I})$ firstly for special partitions.

\begin{lem}\label{lem4.6} (1) When $|I|=0$, $V_0(r,0,\{\lambda_x\}_{x\in I})=1$;

(2) $V_0(r,0,\lambda_x)=1$ if $\lambda_x=0$ and zero otherwise;

(3) $V_0(r,0,\{\lambda_x,\lambda_y\})=1$ if $\lambda_x\sim\lambda_y^{\ast}$ and zero otherwise;

(4) Let $Y(\lambda_y,\omega_s)$ be the set defined in \eqref{4.14}. Then
$$V_0(r,0,\{\omega_s,\lambda_y,\lambda_z\})=\left\{
\begin{array}{llll} 1 &\mbox{when $\lambda^*_z\sim\mu \in Y(\lambda_y,\omega_s)$}\\
0&\mbox{when
$\lambda^*_z\nsim\mu\in Y(\lambda_y,\omega_s)$.}\end{array}\right.$$
Note that, for any $\mu,\,\mu'\in Y(\lambda_y,\omega_s)$, $\mu\sim\mu'\,\Leftrightarrow\,\mu=\mu'$.
\end{lem}
\begin{proof} (1) When $|I|=0$ and $r|d$, recall from \eqref{4.11}, we have
$$V_0(r,d,\{\lambda_x\}_{x\in I})=\frac{1}{r(r+k)^{r-1}}\sum_{\vec v}\Delta(\vec v)\overline{\Delta(\vec v)}$$
where $\vec v=(v_1,v_2,\ldots,v_r)$ runs through the integers $$0=v_r<v_{r-1}<\cdots <v_2<v_1<r+k.$$
Let $\rho=(\rho_1,...,\rho_r)=(r-1,...,1,0)$. By expansion of determinant,
$$\Delta(\vec v)=\sum_{\tau\in \mathfrak{S}_r}\epsilon(\tau)e^{2\pi i\frac{v_1}{r+k}\rho_{\tau(1)}}
\cdot\cdots\cdot e^{2\pi i\frac{v_r}{r+k}\rho_{\tau(r)}}=J(e^{\vec{\rho}})({\rm exp}\,2\pi i\frac{\vec v}{r+k}).$$
Then the same computations in the proof of Proposition \ref{prop4.1} imply that
$$\aligned&\sum_{\vec v}\Delta(\vec v)\overline{\Delta(\vec v)}=\frac{1}{|\mathfrak{S}_{r-1}|}\sum_{t\in T'_k}J(e^{\vec{\rho}})(t)\overline{J(e^{\vec{\rho}})(t)}\\&
=\frac{1}{|\mathfrak{S}_{r-1}|}\sum_{\tau,\sigma\in \mathfrak{S}_r}\epsilon(\tau)\epsilon(\sigma)\sum_{t\in T'_k}e^{\tau(\vec\rho)}(t)\cdot\overline{e^{\sigma(\vec\rho)}(t)}
=r(r+k)^{r-1}\endaligned$$
where $T'_k=\{\,t=\mathrm{diag}(e^{\frac{2\pi i}{r+k}t_1}, \cdots, e^{\frac{2\pi i}{r+k}t_{r-1}},1)\,|\, 0\le t_i<r+k\,\}$.

To prove (2) and (3), we note that (3) implies (2) since $$V_0(r,0,\{\lambda_x\})=V_0(r,0,\{\lambda_x,\lambda_y\})$$
when $\lambda_y=0$. Thus it is enough to show that
$$\sum_{\vec v}
\frac{{\rm exp}\left(2\pi i\left(-\frac{|\lambda_x|+|\lambda_y|}{r(r+k)}\right)\sum_{i=1}^{r}v_i\right)S_{\{\lambda_x,\lambda_y\}}\left({\rm exp}\,2\pi i\frac{\vec v}{r+k}\right)} {\prod_{i<j}\left(2\sin\,\pi \frac{v_i-v_j}{r+k}\right)^{-2}}=r(r+k)^{r-1}$$
when $\lambda_x\sim\lambda_y^{\ast}$ and zero otherwise. Write $\vec v=\rho+\mu$ ($\mu\in W_k$), then
$$S_{\lambda_x}\left({\rm exp}\,2\pi i\frac{\vec v}{r+k}\right)=\frac{\Delta(\vec\lambda_x+\vec\rho)}{\Delta(\vec v)}S_{\mu}\left(\mathrm{exp}2\pi i\frac{\vec\lambda_x+\vec\rho}{r+k} \right)$$
$$S_{\lambda^*_y}\left({\rm exp}\,2\pi i\frac{\vec v}{r+k}\right)=\frac{\Delta(\vec\lambda^*_y+\vec\rho)}{\Delta(\vec v)}S_{\mu}\left(\mathrm{exp}2\pi i\frac{\vec\lambda^*_y+\vec\rho}{r+k} \right)$$
$$S_{\lambda_y}\left(\mathrm{exp}2\pi i\frac{\vec v}{r+k} \right)=\overline{S_{\lambda_y^{\ast}}\left(\mathrm{exp}2\pi i\frac{\vec v}{r+k} \right)}\cdot \mathrm{exp}\left(2\pi i\frac{k}{r+k}\sum_{i=1}^r v_i \right).$$
Thus we have
$$\aligned &\sum_{\vec v}
\frac{e^{2\pi i\left(-\frac{|\lambda_x|+|\lambda_y|}{r(r+k)}\right)\sum_{i=1}^{r}v_i}S_{\{\lambda_x,\lambda_y\}}\left({\rm exp}\,2\pi i\frac{\vec v}{r+k}\right)} {\prod_{i<j}\left(2\sin\,\pi \frac{v_i-v_j}{r+k}\right)^{-2}}=\frac{\Delta(\vec\lambda_x+\vec\rho)\overline{\Delta(\vec\lambda_y^{\ast}+\vec\rho)}}{\mathrm{exp}\left(2\pi i\frac{k}{r+k}|\lambda_y^{\ast}+\rho| \right)}\\&
\sum_{\mu\in W_k}e^{2\pi i\left(\frac{|\lambda^*_y|-|\lambda_x|}{r(r+k)}\right)|\mu+\rho|} S_{\mu}\left(\mathrm{exp}2\pi i\frac{\vec\lambda_x+\vec\rho}{r+k} \right)\cdot
S_{\mu^*}\left(\mathrm{exp}2\pi i\frac{\vec\lambda_y^{\ast}+\vec\rho}{r+k} \right),
\endaligned$$
which equals $1$ when $\lambda_x\sim\lambda_y^{\ast}$ by \eqref{4.3} of Proposition \ref{prop4.1} and otherwise zero if $\lambda_x\nsim\lambda_y^*$ by \eqref{4.4} of Proposition \ref{prop4.1}.

In order to prove (4), by using of \eqref{4.14}, we have
$$S_{\omega_s}\left(\mathrm{exp}2\pi i\frac{\vec v}{r+k} \right)\cdot S_{\lambda_y}\left(\mathrm{exp}2\pi i\frac{\vec v}{r+k} \right)=\sum_{\mu\in Y(\lambda_y,\omega_s)}S_{\mu}\left(\mathrm{exp}2\pi i\frac{\vec v}{r+k} \right).$$
Since $|\mu|=|\lambda_y|+|\omega_s|$ for any $\mu\in Y(\lambda_y,\omega_s)$, we have
 $$\aligned &V_{0}(r,0,\{\omega_s,\lambda_y,\lambda_z\})\\&=\frac{1}{r(r+k)}
\sum_{\vec v}\frac{\mathrm{exp}\left(-2\pi i\frac{|\omega_s|+|\lambda_y|+|\lambda_z|}{r(r+k)}\right)S_{\{\omega_s,\lambda_y,\lambda_z\}}\left(\mathrm{exp}2\pi i\frac{\vec v}{r+k}\right)}{\prod_{i<j}(2sin \pi\frac{v_i-v_j}{r+k})^{-2}}\\&
=\frac{1}{r(r+k)}\sum_{\vec v}\sum_{\mu\in Y(\lambda_y,\omega_s)}\frac{\mathrm{exp}\left(-2\pi i\frac{|\mu|+|\lambda_z|}{r(r+k)}\right)S_{\{\mu,\lambda_z\}}\left( \mathrm{exp}2\pi i\frac{\vec v}{r+k}\right)}{\prod_{i<j}(2sin \pi\frac{v_i-v_j}{r+k})^{-2}}
\\&
=\sum_{\mu\in Y(\lambda_y,\omega_s)}V_0(r,0,\{\mu, \lambda_z\}),\endaligned$$
which and (3) imply (4).
\end{proof}

Let $\mathcal{U}_{\mathbb{P}^1}(r, 0,\{\lambda_x\}_{x\in I})$ be the moduli space of semi-stable parabolic bundles of rank $r$ and degree $0$ on $\mathbb{P}^1$ with parabolic structures given by $\{\lambda_x\}_{x\in I}$. Recall condition \eqref{3.1} (the necessary condition to define theta line bundle on $\mathcal{U}_{\mathbb{P}^1}(r, 0,\{\lambda_x\}_{x\in I})$): $\frac{\sum_{x\in I}|\lambda_x|}{r}\in \mathbb{Z}$. We will assume this condition (in case it is needed), otherwise $V_0(r, 0,\{\lambda_x\}_{x\in I})=0$ and we can define $D_0(r, 0,\{\lambda_x\}_{x\in I})=0$.

\begin{lem}\label{lem4.7} (1) When $|I|=0$, $\mathcal{U}_{\mathbb{P}^1}(r, 0,\{\lambda_x\}_{x\in I})$ consists one point;

(2) When $|I|=1$, $\mathcal{U}_{\mathbb{P}^1}(r, 0,\lambda_x)$ consists one point if $\lambda_x=0$ and is empty otherwise;

(3) When $|I|=2$, $\mathcal{U}_{\mathbb{P}^1}(r, 0,\{\lambda_x,\lambda_y\})$ consists one point if $\lambda_x\sim\lambda_y^*$ and is empty otherwise;

(4) When $|I|=3$, $\mathcal{U}_{\mathbb{P}^1}(r, 0,\{\omega_s,\lambda_y,\lambda_z\})$ ($1\le s\le r-1$) consists one point if $\lambda_z^{\ast}\sim\mu \in Y(\lambda_y,\omega_s)$ and is empty otherwise.
\end{lem}

\begin{proof} (1) is clear. For other statements, recall (Notation \ref{nota4.2}) the parabolic structure of $E$ at $x\in I$ determined by $\lambda_x$ is given by a flag
$$E_x=Q_{l_x+1}(E)_x\twoheadrightarrow
Q_{l_x}(E)_x\twoheadrightarrow\cdots\cdots\twoheadrightarrow
Q_1(E)_x\twoheadrightarrow Q_0(E)_x=0$$
and $0\leq a_1(x)<a_2(x)<\cdots<a_{l_x+1}(x)\le k$ (here we have to include the case of $a_{l_x+1}(x)- a_1(x)=k$ since Hecke modifications) such that
$$\lambda_x=(\,\overbrace{k-a_1(x),\ldots,k-a_1(x)}^{n_1(x)}\,\ldots,\, \overbrace{k-a_{l_x+1}(x),\ldots,k-a_{l_x+1}(x)}^{n_{l_x+1}(x)}\,)$$
where $n_i(x)=r_i(x)-r_{i-1}(x)$ ($1\le i\le l_x+1$) and $r_i(x)={\rm dim}\,Q_i(E)_x$.
For any subsheaf $F\subset E$, let $Q_i(E)_x^F\subset Q_i(E)_x$ be the image of $F_x$ and $n_i^F(x)=r^F_i(x)-r^F_{i-1}(x)$, $r^F_i(x)={\rm dim}\,Q_i(E)_x^F$. Let
$$\aligned &{\rm par}{\rm deg}(E):={\rm deg}(E)+\frac{1}{k}\sum_{x\in
I}\sum^{l_x+1}_{i=1}a_i(x)n_i(x),\\&{\rm par}{\rm deg}(F):={\rm deg}(F)+\frac{1}{k}\sum_{x\in
I}\sum^{l_x+1}_{i=1}a_i(x)n^F_i(x).\endaligned$$
Then $E$ is called semistable (resp., stable) for $\omega=(k, \{\vec n(x),\,\,\vec a(x)\}_{x\in I})$ if for any
nontrivial sub-bundle $E'\subset E$,
one has
$${\rm par}\mu(E'):=\frac{{\rm par}{\rm deg}(E')}{r(E')}\leq
\frac{{\rm par}{\rm deg}(E)}{r}:={\rm par}\mu(E)\,\,(\text{resp., }<).$$

To show (2), when $\lambda_x$ is nontrivial (i.e., $0<n_1(x)<r$), we have
$${\rm par}\mu(E)\le \frac{1}{r}\sum^{l_x+1}_{i=1}\frac{a_{l_x+1}(x)-(l_x+1-i)}{k}n_i(x)<\frac{a_{l_x+1}(x)}{k}.$$
Thus any semistable parabolic bundle must have $E=\sO_{\mathbb{P}^1}^{\oplus r}$ and the evaluation map
${\rm H}^0(\mathbb{P}^1,E)\to E_x$ is an isomorphism. Then there is a line bundle $F\subset E$ of degree zero such that
$n^F_{l_x+1}(x)=1$, which implies
$$ {\rm par}\mu(F)=\frac{a_{l_x+1}(x)}{k}>{\rm par}\mu(E)$$
and $\mathcal{U}_{\mathbb{P}^1}(r, 0,\lambda_x)$ is empty.

To prove (3), we consider firstly the case of $\lambda_x\sim\lambda_y^*$, then
$$l_y=l_x, \quad n_i(y)=n_{l_x-i+2}(x),\quad a_i(y)+a_{l_x-i+2}(x)=a$$
and ${\rm par}\mu(E)=\frac{a}{k}$. To show that $\mathcal{U}_{\mathbb{P}^1}(r, 0,\{\lambda_x,\lambda_y\})$ consists of one point, recall that it is a GIT quotient
$\psi: \sR^{ss}_{\{\lambda_x,\lambda_y\}} \to \mathcal{U}_{\mathbb{P}^1}(r, 0,\{\lambda_x,\lambda_y\})$ where $\sR^{ss}_{\{\lambda_x,\lambda_y\}}\subset\sR$
is an (maybe empty) open set of an irreducible quasi-projective variety $\sR$. Let $\sR^0\subset\sR$ be the (non-empty) open set of parabolic bundles $E$ with a trivial
underline bundle. We will show that $\sR^0\cap \sR^{ss}_{\{\lambda_x,\lambda_y\}}$ is non-empty and $\psi(\sR^0\cap \sR^{ss}_{\{\lambda_x,\lambda_y\}})$ is a point, which implies that $\mathcal{U}_{\mathbb{P}^1}(r, 0,\{\lambda_x,\lambda_y\})$ consists of one point. Clearly, the following semistable parabolic bundle
\ga{4.15}{\bigoplus^{l_y+1}_{i=1}(\sO_{\mathbb{P}^1},\,\{a_{l_x-i+2}(x),\,a_i(y)\})^{\oplus n_i(y)}}
defines a point of $\mathcal{U}_{\mathbb{P}^1}(r, 0,\{\lambda_x,\lambda_y\})$, where $(\sO_{\mathbb{P}^1},\,\{a_{l_x-i+2}(x),\,a_i(y)\})$ is a parabolic line bundle with weight $\{a_{l_x-i+2}(x),\,a_i(y)\}$. On the other hand, any
$E\in \sR^0\cap \sR^{ss}_{\{\lambda_x,\lambda_y\}}$ is $s$-equivalent to the parabolic bundle defined in \eqref{4.15}. In fact,
since ${\rm H}^0(\mathbb{P}^1,E)\to E_x$ is an isomorphism, there is a $\sL_1=\sO_{\mathbb{P}^1}\subset E$ such that $n^{\sL_1}_{l_x+1}(x)=1$. Then the semistability of
$E$ implies $n^{\sL_1}_1(y)=1$ and $\sL_1=(\sO_{\mathbb{P}^1},\,\{a_{l_x+1}(x),\,a_1(y)\})$ has
${\rm par}\mu(\sL_1)={\rm par}\mu(E)$. Let $E'=E/\sL_1$ be the quotient parabolic bundle, then $E$ is $s$-equivalent to $\sL_1\oplus E'$ where the parabolic structures of $E'$ are defined by partitions
$$\lambda'_x=(\,\overbrace{k-a_1(x),\ldots,k-a_1(x)}^{n_1(x)}\,\ldots,\, \overbrace{k-a_{l_x+1}(x),\ldots,k-a_{l_x+1}(x)}^{n_{l_x+1}(x)-1}\,),$$
$$\lambda'_y=(\,\overbrace{k-a_1(y),\ldots,k-a_1(y)}^{n_1(y)-1}\,\ldots,\, \overbrace{k-a_{l_y+1}(y),\ldots,k-a_{l_y+1}(y)}^{n_{l_y+1}(y)}\,).$$
Clearly, $\lambda_x'\sim\lambda_y^{'*}$ and, by induction of $r(E)$, $E'$ is $s$-equivalent to
$$ (\sO_{\mathbb{P}^1},\,\{a_{l_x+1}(x),\,a_1(y)\})^{\oplus (n_1(y)-1)}\oplus\bigoplus^{l_y+1}_{i=2}(\sO_{\mathbb{P}^1},\,\{a_{l_x-i+2}(x),\,a_i(y)\})^{\oplus n_i(y)},$$
thus $E$ is $s$-equivalent to the parabolic bundle defined in \eqref{4.15}.

To prove $\mathcal{U}_{\mathbb{P}^1}(r, 0,\{\lambda_x,\lambda_y\})$ is empty when $\lambda_x\nsim\lambda_y^*$, it is enough to prove
$\sR^0\cap \sR^{ss}_{\{\lambda_x,\lambda_y\}}$ is empty. Let $E\in \sR^0$, recall its flag at $x$ and $y$
$$E_x=Q_{l_x+1}(E)_x\twoheadrightarrow
Q_{l_x}(E)_x\twoheadrightarrow\cdots\cdots\twoheadrightarrow
Q_1(E)_x\twoheadrightarrow Q_0(E)_x=0,$$
$$E_y=Q_{l_y+1}(E)_y\twoheadrightarrow
Q_{l_y}(E)_y\twoheadrightarrow\cdots\cdots\twoheadrightarrow
Q_1(E)_y\twoheadrightarrow Q_0(E)_y=0$$ and $r_i(x):={\rm dim}(Q_i(E)_x)$, $r_i(y):={\rm dim}(Q_i(E)_y)$. For $1\leq m<r$, let
$m_x=\mathrm{max}\{i\mid r_i(x)\leq m \}$, $m_y=\mathrm{max}\{i\mid  r_i(y)\le m\}$ and
$$d_x(m):=\sum_{i=1}^{m_x}\frac{a_i(x)}{k}n_i(x)+(m-r_{m_x}(x))\frac{a_{m_x+1}(x)}{k},$$
$$\bar{d}_x(m):=\sum_{i=1}^{l_x+1}\frac{a_i(x)}{k}n_i(x)-d_x(m),\quad \text{$d_y(m)$, $\bar{d}_y(m)$ are defined similarly.}$$
\textbf{If $\lambda_x\nsim\lambda_y^*$,
we claim that there is an integer $1\le m<r$ such that either $\frac{d_x(m)+\bar{d}_y(m')}{m}>{\rm par}\mu(E)$ or $\frac{\bar{d}_x(m)+d_y(m')}{m'}>{\rm par}\mu(E)$
where $m'=r-m$}. If the claim is true, without loss of generality, we assume
$$\frac{d_x(m)+\bar{d}_y(m')}{m}>{\rm par}\mu(E)$$
holds for some $1\le m<r$. Note $K_{m'_y+1}:={\rm ker}(Q_{m'_y+1}(E)_y\twoheadrightarrow Q_{m'_y}(E)_y)$ has
${\rm dim}(K_{m'_y+1})=r_{m'_y+1}(y)-r_{m'_y}(y)\ge r_{m'_y+1}(y)-m'$ and, for any subspace $W\subset K_{m'_y+1}$ of
dimension $r_{m'_y+1}(y)-m'$, there is a subbundle $F=\sO_{\mathbb{P}^1}^{\oplus m}\subset E$ such that $Q_{m'_y+1}(E)^F_y=W$.
Then $n^F_{m'_y+1}(y)=r_{m'_y+1}(y)-m'$, $n^F_i(y)=n_i(y)$ ($m'_y+2\le i\le l_y+1$), $n^F_i(y)=0$ ($1\le i\le m'_y$), $\bar{d}_y(m')=\sum^{l_y+1}_{i=1}\frac{a_i(y)}{k}n^F_i(y)$ and  $$\aligned&\sum^{l_x+1}_{i=1}\frac{a_i(x)}{k}n^F_i(x)-d_x(m)=\sum^{l_x+1}_{i=m_x+1}\frac{a_i(x)}{k}n^F_i(x)-\sum^{m_x}_{i=1}\frac{a_i(x)}{k}(n_i(x)-n^F_i(x))
\\&-(m-r_{m_x}(x))\frac{a_{m_x+1}(x)}{k}\ge\left(\sum^{l_x+1}_{i=1}n^F_i(x)-m\right)\frac{a_{m_x+1}(x)}{k}=0,\endaligned$$
which imply ${\rm par}\mu(F)\ge \frac{d_x(m)+\bar{d}_y(m')}{m}>{\rm par}\mu(E)$. Thus $E$ is not semistable.

To prove the claim, if both $\frac{d_x(m)+\bar{d}_y(m')}{m}\le {\rm par}\mu(E)$ and $\frac{\bar{d}_x(m)+d_y(m')}{m'}\le{\rm par}\mu(E)$ hold for
all $1\le m<r$, we will show
\ga{4.16}{\frac{a_i(x)}{k}+\frac{a_{l_y-i+2}(y)}{k}={\rm par}\mu(E),\quad n_i(x)=n_{l_y-i+2}(y)}
for $1\le i\le {\rm min}\{l_x+1,\,l_y+1\}$, which imply $\lambda_x\sim\lambda_y^*$. Indeed, if both $\frac{d_x(m)+\bar{d}_y(m')}{m}\le {\rm par}\mu(E)$ and $\frac{\bar{d}_x(m)+d_y(m')}{m'}\le{\rm par}\mu(E)$ hold for
all $1\le m<r$, we must have the equalities (for all $1\le m<r$)
\ga{4.17}{\frac{d_x(m)+\bar{d}_y(m')}{m}={\rm par}\mu(E),\quad \frac{\bar{d}_x(m)+d_y(m')}{m'}={\rm par}\mu(E)}
since $d_x(m)+\bar{d}_y(m')+\bar{d}_x(m)+d_y(m')=(m+m'){\rm par}\mu(E)$.
We will prove \eqref{4.16} by taking different $1\le m<r$ in \eqref{4.17}. To check \eqref{4.16} for $i=1$ firstly. Take $m=1$, we have $d_x(m)=\frac{a_1(x)}{k}$, $\bar{d}_y(m')=\frac{a_{l_y+1}(y)}{k}$
and $\frac{a_1(x)}{k}+\frac{a_{l_y+1}(y)}{k}={\rm par}\mu(E)$. Take $m=r_1(x)$ and
$m'=r-r_1(x)$, then $d_x(m)=\frac{a_1(x)}{k}r_1(x)$ and
$$\aligned&\bar{d}_y(m')=\sum^{l_y+1}_{i=m'_y+2}\frac{a_i(y)}{k}n_i(y)+(r_{m'_y+1}(y)-m')\frac{a_{m'_y+1}(y)}{k}\\&\le
\frac{a_{l_y+1}(y)}{k}n_{l_y+1}(y)+(r_{m'_y+1}(y)-m')\frac{a_{m'_y+1}(y)}{k}\,\,(\text{if $m'_y<l_y$})\\&<
\frac{a_{l_y+1}(y)}{k}n_{l_y+1}(y)+(r_{l_y}(y)-m')\frac{a_{l_y+1}(y)}{k}=\frac{a_{l_y+1}(y)}{k}r_1(x)
\endaligned$$
which implies ${\rm par}\mu(E)=\frac{d_x(m)+\bar{d}_y(m')}{m}<\frac{a_1(x)}{k}+\frac{a_{l_y+1}(y)}{k}={\rm par}\mu(E)$. Thus we must have
$m'_y=l_y$, i.e. $r_{l_y}(y)\le m'=r-r_1(x)$, which means $n_{l_y+1}(y)=r-r_{l_y}(y)\ge r_1(x)=n_1(x)$. In fact, $n_1(x)=n_{l_y+1}(y)$.
Otherwise, take $m=n_1(x)+1\le n_{l_y+1}(y)$ (which implies $m'_y=l_y$), then $d_x(m)=\frac{a_1(x)}{k}n_1(x)+\frac{a_2(x)}{k}$,
$\bar{d}_y(m')=\frac{a_{l_y+1}(x)}{k}m$ and (by $a_1(x)<a_2(x)$) we get contradiction:
${\rm par}\mu(E)=\frac{d_x(m)+\bar{d}_y(m')}{m}<\frac{a_1(x)}{k}+\frac{a_{l_y+1}(y)}{k}={\rm par}\mu(E).$
Assume $1<i_0\le {\rm min}\{l_x+1,\,l_y+1\}$ such that \eqref{4.16}
holds for all $i<i_0$, we show \eqref{4.16}
holds for $i=i_0$. Take $m=r_{i_0-1}(x)+1$, then $m'=r-r_{i_0-1}(x)-1=r_{l_y-i_0+2}(y)-1$, $m'_y=l_y-i_0+1$, $d_x(m)=\sum^{i_0-1}_{i=1}\frac{a_i(x)}{k}n_i(x)+\frac{a_{i_0}(x)}{k}$ and $\bar{d}_y(m')=\sum^{l_y+1}_{i=l_y-i_0+3}\frac{a_i(y)}{k}n_i(y)+\frac{a_{l_y-i_0+2}(y)}{k}$.
By \eqref{4.17} and $\sum^{i_0-1}_{i=1}\frac{a_i(x)}{k}n_i(x)+\sum^{l_y+1}_{i=l_y-i_0+3}\frac{a_i(y)}{k}n_i(y)=r_{i_0-1}(x){\rm par}\mu(E)$,
we have $\frac{a_{i_0}(x)}{k}+\frac{a_{l_y-i_0+2}(y)}{k}={\rm par}\mu(E)$. If $n_{i_0}(x)>n_{l_y-i_0+2}(y)$, take $m=r_{i_0}(x)$, then $m'=r_{l_y-i_0+2}(y)-n_{i_0}(x)$ and $m'_y\le l_y-i_0$ (otherwise $n_{i_0}(x)\le n_{l_y-i_0+2}(y)$), which imply a contradiction:
$$\aligned &d_x(m)+\bar{d}_y(m')=\sum^{i_0-1}_{i=1}\frac{a_i(x)+a_{l_y-i+2}(y)}{k}n_i(x)+\frac{a_{i_0}(x)}{k}n_{i_0}(x)
+\\&\sum^{l_y-i_0+2}_{i=m'_y+2}\frac{a_i(y)}{k}n_i(y)+(r_{m'_y+1}(y)-m')\frac{a_{m'_y+1}(y)}{k}< r_{i_0-1}(x){\rm par}\mu(E)
+\\&\frac{a_{i_0}(x)}{k}n_{i_0}(x)+\frac{a_{i_0}(x)+a_{l_y-i_0+2}(y)}{k}n_{i_0}(x)=r_{i_0}(x){\rm par}\mu(E).
\endaligned$$ If $n_{i_0}(x)<n_{l_y-i_0+2}(y)$, take $m=r_{i_0}(x)+1$, then
$m'=r-r_{i_0}(x)-1=r_{l_y-i_0+2}(y)-(n_{i_0}(x)+1)\ge r_{l_y-i_0+1}(y)$. Thus $m'_y=l_y-i_0+1$ and
$d_x(m)+\bar{d}_y(m')=\sum^{i_0}_{i=1}\frac{a_i(x)+a_{l_y-i+2}(y)}{k}n_i(x)+\frac{a_{i_0+1}(x)}{k}+\frac{a_{l_y-i_0+2}(y)}{k}=
(r_{i_0}(x)+1){\rm par}\mu(E)+\frac{a_{i_0+1}(x)}{k}-\frac{a_{i_0}(x)}{k}$, which is a contradiction since $a_{i_0+1}(x)>a_{i_0}(x)$.
We must have $n_{i_0}(x)=n_{l_y-i_0+2}(y)$.

To prove (4), let $\mu=(\,\,\overbrace{\mu_1,\ldots,\mu_1}^{n_1},\cdots,\overbrace{\mu_{l+1},\ldots,\mu_{l+1}}^{n_{l+1}}\,\,)\in Y(\lambda_y,\omega_s)$ and recall
that parabolic structure at $x\in \mathbb{P}^1$ determined by $\omega_s$ is
$$E_x=Q_2(E)_x\twoheadrightarrow Q_1(E)_x\twoheadrightarrow Q_0(E)_x=0$$
with $n_1(x)=s$, $n_2(x)=r-s$, $a_1(x)=k-1$ and $a_2(x)=k$. For any $E\in \sR^0$, it is easy to compute that
${\rm par}\mu(E)=2-\frac{1}{k\cdot r}(s+|\lambda_y|-|\lambda_z^*|)$.

\textbf{When $\lambda_z^*\sim \mu$}, we have $l=l_z$, $n_i=n_{l_z-i+2}(z)$ and there is a constant $a\in \mathbb{Z}$ such that
$\mu_i-a_{l_z-i+2}(z)=a$ ($1\le i\le l_z+1$), which implies ${\rm par}\mu(E)=2-\frac{a}{k}$. We are going to prove that $\sR^0\cap \sR^{ss}_{{\{\omega_s,\lambda_y,\lambda_z\}}}$
is nonempty and any $E\in \sR^0\cap \sR^{ss}_{{\{\omega_s,\lambda_y,\lambda_z\}}}$ is $s$-equivalent to a direct sum of parabolic line bundles.
If $\mu_1=k-a_1(y)+1$, let $\lambda_i(y):=k-a_i(y)$,
$$\aligned&\lambda_y'=(\,\,\overbrace{\lambda_1(y),\ldots,\lambda_1(y)}^{n_1(y)-1},\,\,\overbrace{\lambda_1(y),\ldots,\lambda_1(y)}^{n_2(y)},\,\,\cdots,\,\,
\overbrace{\lambda_{l_y+1}(y),\ldots,\lambda_{l_y+1}(y)}^{n_{l_y+1}(y)}\,\,),\\
&\lambda'^{*}_z=(\,\,\overbrace{a_{l_z+1}(z),\ldots,a_{l_z+1}(z)}^{n_{l_z+1}(z)-1},\,\,\overbrace{a_{l_z}(z),\ldots,a_{l_z}(z)}^{n_{l_z}(z)},\,\,\cdots,\, \,\overbrace{a_1(z),\ldots,a_1(z)}^{n_1(z)}\,\,)
\endaligned$$
and $\sL=(\sO_{\mathbb{P}^1},\{k-1,\,a_1(y),\,a_{l_z+1}(z)\})$. Note that
$$\lambda'^{*}_z\sim \mu'=(\,\,\overbrace{\mu_1,\ldots,\mu_1}^{n_1-1},\cdots,\overbrace{\mu_{l+1},\ldots,\mu_{l+1}}^{n_{l+1}}\,\,)\in Y(\lambda'_y,\omega_{s-1}),$$
we have $\sR^0\cap \sR^{ss}_{{\{\omega_{s-1},\lambda'_y,\lambda'_z\}}}\neq\emptyset$ by induction. Let
$$E'=(\,\sO_{\mathbb{P}^1}^{\oplus(r-1)},\,\{\omega_{s-1},\lambda'_y,\lambda'_z\})\in \sR^0\cap \sR^{ss}_{{\{\omega_{s-1},\lambda'_y,\lambda'_z\}}}$$
and $E=\sL\oplus E'$, then $E\in \sR^0\cap \sR^{ss}_{{\{\omega_s,\lambda_y,\lambda_z\}}}$ since
$${\rm par}\mu(\sL)=\frac{k-1+a_1(y)+a_{l_z+1}(z)}{k}={\rm par}\mu(E')=2-\frac{a}{k}.$$
Conversely, for any $E\in \sR^0\cap \sR^{ss}_{{\{\omega_s,\lambda_y,\lambda_z\}}}$,
there is a $\sL=\sO_{\mathbb{P}^1}\subset E$ such that $n^{\sL}_{l_z+1}(z)=1$. The semi-stability of $E$ implies $n^{\sL}_1(x)=1$
and $n_1^{\sL}(y)=1$. Then $\sL=(\sO_{\mathbb{P}^1}, \{k-1, a_1(y), a_{l_z+1}(z)\})$ is a parabolic sub-bundle of $E=(\sO_{\mathbb{P}^1}^{\oplus r}, \{\omega_s,\lambda_y,\lambda_z\})$ with ${\rm par}\mu(\sL)={\rm par}\mu(E)$ and its quotient parabolic bundle $E'=(\sO_{\mathbb{P}^1}^{\oplus (r-1)},\{\omega_{s-1},\lambda'_y,\lambda'_z\})$ is semi-stable. By induction, $E$ is $s$-equivalent to direct sum of parabolic line bundles.

If $\mu_1=k-a_1(y)$, we have $a_1(y)+a_{l_z+1}(z)=k-a$,
$$\lambda'^{*}_z\sim \mu'=(\,\,\overbrace{\mu_1,\ldots,\mu_1}^{n_1-1},\cdots,\overbrace{\mu_{l+1},\ldots,\mu_{l+1}}^{n_{l+1}}\,\,)\in Y(\lambda'_y,\omega_s)$$
and there exists $E'=(\,\sO_{\mathbb{P}^1}^{\oplus(r-1)},\,\{\omega_s,\lambda'_y,\lambda'_z\})\in \sR^0\cap \sR^{ss}_{{\{\omega_s,\lambda'_y,\lambda'_z\}}}$ by
induction of the rank. Then $E'\oplus \sL\in \sR^0\cap \sR^{ss}_{{\{\omega_s,\lambda_y,\lambda_z\}}}$ where
$$\sL=(\sO_{\mathbb{P}^1}, \{k, a_1(y), a_{l_z+1}(z)\})$$
with ${\rm par}\mu(\sL)=\frac{k+a_1(y)+a_{l_z+1}(z)}{k}=2-\frac{a}{k}={\rm par}\mu(E)={\rm par}\mu(E')$. On the other hand, $\forall\,\,E\in \sR^0\cap \sR^{ss}_{{\{\omega_s,\lambda_y,\lambda_z\}}}$, let $E'=\sO_{\mathbb{P}^1}^{\oplus(r-1)}\subset E$ such that
$$E'_y\supset K_y={\rm ker}\{E_y=Q_{l_y+1}(E)_y\to Q_1(E)_y\}.$$ Then the induced flag
$E'_y\twoheadrightarrow Q_{l_y}(E)_y^{E'}\twoheadrightarrow\cdots\twoheadrightarrow Q_1(E)_y^{E'}\to 0$ must have
${\rm dim}(Q_i(E)_y^{E'})={\rm dim}(Q_i(E)_y)-1$ ($1\le i\le l_y+1$). Let
$$0\to E'\to E\to\sL\to 0$$ and $\sL=(\sO_{\mathbb{P}^1}, \{a_i(x), a_1(y), a_j(z)\})$ be the induced parabolic quotient line bundle.
Then ${\rm par}\mu(\sL)\ge {\rm par}\mu(E)=2-\frac{a}{k}$ by semi-stability of $E$. But
${\rm par}\mu(\sL)=\frac{a_i(x)+a_1(y)+a_j(z)}{k}=2-\frac{a}{k}+\frac{a_i(x)-k+a_j(z)-a_{l_z+1}(z)}{k}\le 2-\frac{a}{k}$, the equality
holds if and only if $a_i(x)=k$ and $a_j(z)=a_{l_z+1}(z)$. Thus $E$ is $s$-equivalent to $E'\oplus (\sO_{\mathbb{P}^1}, \{k, a_1(y), a_{l_z+1}(z)\})$ where
$$E'=(\,\sO_{\mathbb{P}^1}^{\oplus(r-1)},\,\{\omega_s,\lambda'_y,\lambda'_z\})\in \sR^0\cap \sR^{ss}_{{\{\omega_s,\lambda'_y,\lambda'_z\}}}.$$

\textbf{When $\lambda_z^*\nsim \mu$ for any $\mu\in Y(\lambda_y,\omega_s)$}, we prove $\sR^0\cap \sR^{ss}_{{\{\omega_s,\lambda_y,\lambda_z\}}}=\emptyset.$
In fact, if there is a $E\in \sR^0\cap \sR^{ss}_{{\{\omega_s,\lambda_y,\lambda_z\}}}$, we will prove that there exists $\mu\in Y(\lambda_y,\omega_s)$ such that
$\lambda_z^*\sim \mu$. Let $a=\frac{s+|\lambda_y|-|\lambda_z^*|}{r}\in \mathbb{Z}$, then ${\rm par}\mu(E)=2-\frac{a}{k}$ and $k-1+a_1(y)+a_{l_z+1}(z)\le 2k-a$. Otherwise, let $\sL=\sO_{\mathbb{P}^1}\subset E$ be a sub-bundle such that $n^{\sL}_{l_z+1}(z)=1$, which implies
$${\rm par}\mu(\sL)=\frac{a_i(x)+a_j(y)+a_{l_z+1}(z)}{k}\ge\frac{a_1(x)+a_1(y)+a_{l_z+1}(z)}{k}>2-\frac{a}{k}.$$
Similarly, if $k+a_1(y)+a_{l_z+1}(z)<2k-a$, we find a parabolic quotient line bundle $\sL$ with ${\rm par}\mu(\sL)=\frac{a_i(x)+a_1(y)+a_j(z)}{k}\le\frac{k+a_1(y)+a_{l_z+1}(z)}{k}<2-\frac{a}{k}.$ Thus $k-1+a_1(y)+a_{l_z+1}(z)\le 2k-a\le k+a_1(y)+a_{l_z+1}(z)$, which implies either $\frac{k-1+a_1(y)+a_{l_z+1}(z)}{k}=2-\frac{a}{k}$ or $\frac{k+a_1(y)+a_{l_z+1}(z)}{k}=2-\frac{a}{k}$. Then we have
either $\sR^0\cap \sR^{ss}_{{\{\omega_{s-1},\lambda'_y,\lambda'_z\}}}\neq\emptyset$ or $\sR^0\cap \sR^{ss}_{{\{\omega_s,\lambda'_y,\lambda'_z\}}}\neq\emptyset$ (see proof of the case when $\lambda_z^*\sim\mu$). By induction, we have either $\lambda^{'*}_z\sim\mu'\in Y(\lambda_y',\omega_{s-1})$ or $\lambda^{'*}_z\sim\mu'\in Y(\lambda_y',\omega_{s})$, which imply $\lambda^*_z\sim\mu\in Y(\lambda_y,\omega_{s})$.

\end{proof}

\begin{prop}\label{prop4.8} $D_0(r,d,\{\lambda_x\}_{x\in I})=V_0(r,d,\{\lambda_x\}_{x\in I})$ if $|I|\le 3$.
\end{prop}

\begin{proof} We treat firstly the case $r|d$. By Lemma \ref{4.6} and Lemma \ref{4.7}, we can assume $d=0$ and $|I|=3$. For any partition $\lambda=(\lambda_1,...,\lambda_r)$, let $$s(\lambda)={\rm max}\{\,i\,|\, \lambda_i-\lambda_{i+1}>0\,\},\quad m(\lambda)=\sum^{s(\lambda)}_{i=1}\lambda_i$$
and $s(\lambda)=0$ if $\lambda_1=\lambda_2=\cdots=\lambda_r$. When $s(\lambda_x)=0$, $\lambda_x$ defines trivial parabolic structure at $x\in I$ and the proof reduces to the
case of $|I|=2$. Thus, to prove $D_0(r, 0,\{\lambda_x,\lambda_y,\lambda_z\})=V_0(r, 0,\{\lambda_x,\lambda_y,\lambda_z\})$, we can assume $s(\lambda_x)>0$. We will prove it by
induction of $m(\lambda_x)-s(\lambda_x)$.

When $m(\lambda_x)-s(\lambda_x)=0$, $\lambda_x$ must be $\omega_{s(\lambda_x)}$ and the equality
\ga{4.18}{D_0(r, 0,\{\lambda_x,\lambda_y,\lambda_z\})=V_0(r, 0,\{\lambda_x,\lambda_y,\lambda_z\})}
holds by (4) of Lemma \ref{4.6} and Lemma \ref{4.7}. Assume the equality \eqref{4.18} holds for any $\lambda_y$ and $\lambda_z$ when $m(\lambda_x)-s(\lambda_x)<N$.
For any $\lambda_x$ with $m(\lambda_x)-s(\lambda_x)=N$, let $\lambda'_x=\lambda_x-\omega_{s(\lambda_x)}$, then
${\mu\in Y(\lambda'_x,\omega_{s(\lambda_x)})-\{\lambda_x\}}$ is obtained by adding $t<s(\lambda_x)$ boxes to the first $s(\lambda_x)$ rows and other $s(\lambda_x)-t$ boxes respectively to $(s(\lambda_x)+1)$-th row, ..., $(2s(\lambda_x)-t)$-th row. Thus $m(\mu)-s(\mu)=m(\lambda_x)-s(\lambda_x)-(s(\lambda_x)-t)<N$. Then, by
the recurrence relation \eqref{3.9} and (4) of Lemma \ref{lem4.7}, we have
$$\aligned &D_0(r,0,\{\omega_{s(\lambda_x)},\lambda'_x,\lambda_y,\lambda_z\})\\&=\sum_{\mu^{\ast}\in W_k'}D_0(r,0,\{\omega_{s(\lambda_x)},\lambda'_x,\mu^{\ast}\})\cdot D_0(r,0,\{\mu,\lambda_y,\lambda_z\})\\&=\sum_{\mu\in Y(\lambda'_x,\omega_{s(\lambda_x)})}D_0(r,0,\{\mu,\lambda_y,\lambda_z\})\\&=D_0(r,0,\{\lambda_x,\lambda_y,\lambda_z\})+
\sum_{\mu \in Y(\lambda'_x,\omega_{s(\lambda_x)})\setminus\{\lambda_x\}} V_0(r,0,\{\mu,\lambda_y,\lambda_z\}).\endaligned$$
By Lemma \ref{lem4.6}, Lemma \ref{4.7} and recurrence relation \eqref{3.9} again, we have
$$\aligned&D_0(r,0,\{\omega_{s(\lambda_x)},\lambda_y,\lambda'_x,\lambda_z\})\\&=\sum_{\mu\in W'_k}D_0(r,0,\{\omega_{s(\lambda_x)},\lambda_y,\mu\})\cdot D_0(r,0,\{\lambda'_x,\lambda_z,\mu^{\ast}\})
\\&=\sum_{\mu\in W'_k}V_0(r,0,\{\omega_{s(\lambda_x)},\lambda_y,\mu\})\cdot V_0(r,0,\{\lambda'_x,\lambda_z,\mu^{\ast}\})\\&=V_0(r,0,\{\omega_{s(\lambda_x)},\lambda'_x,\lambda_y,\lambda_z\})\endaligned$$
where $D_0(r,0,\{\lambda'_x,\lambda_z,\mu^{\ast}\})
=V_0(r,0,\{\lambda'_x,\lambda_z,\mu^{\ast}\})$ by induction assumption (since either $m(\lambda'_x)-s(\lambda'_x)<N$ or $s(\lambda'_x)=0$). Thus
$$\aligned &D_0(r,0,\{\lambda_x,\lambda_y,\lambda_z\})+
\sum_{\mu \in Y(\lambda'_x,\omega_{s(\lambda_x)})\setminus\{\lambda_x\}} V_0(r,0,\{\mu,\lambda_y,\lambda_z\})\\&=V_0(r,0,\{\omega_{s(\lambda_x)},\lambda'_x,\lambda_y,\lambda_z\})\\&=\frac{1}{r(r+k)}
\sum_{\vec v}\frac{\mathrm{exp}\left(-2\pi i\frac{|\omega_{s(\lambda_x)}|+|\lambda'_x|+|\lambda_y|+|\lambda_z|}{r(r+k)}\right)S_{\{\omega_{s(\lambda_x)},\lambda'_x,\lambda_y,\lambda_z\}}\left(\mathrm{exp}2\pi i\frac{\vec v}{r+k}\right)}{\prod_{i<j}(2sin \pi\frac{v_i-v_j}{r+k})^{-2}}\\&
=\frac{1}{r(r+k)}\sum_{\vec v}\sum_{\mu\in Y(\lambda'_x,\omega_{s(\lambda_x)})}\frac{\mathrm{exp}\left(-2\pi i\frac{|\mu|+|\lambda_y|+|\lambda_z|}{r(r+k)} \right)S_{\{\mu,\lambda_y,\lambda_z\}}\left( \mathrm{exp}2\pi i\frac{\vec v}{r+k}\right)}{\prod_{i<j}(2sin \pi\frac{v_i-v_j}{r+k})^{-2}}
\\&
=V_0(r,0,\{\lambda_x,\lambda_y,\lambda_z\})+\sum_{\mu\in Y(\lambda'_x,\omega_{s(\lambda_x)})\setminus\{\lambda_x\}}V_0(r,0,\{\mu,\lambda_y,\lambda_z\}),\endaligned$$
which implies equality \eqref{4.18} and finishes the proof of case $r|d$.

If $r\nmid d$, we can assume $0<d<r$. Let $\{\lambda_x\}_{x\in I}=\{\lambda_x\}_{x\in I'}\cup \{\lambda_z\}$. Then, by Lemma \ref{lem3.9} and Lemma \ref{lem3.10}
(or see proof of Theorem \ref{thm3.12} in terms of partitions), we have
$$D_g(r,d,\{\lambda_x\}_{x\in I})=D_g(r, 0, \{\lambda_x\}_{x\in I'}\cup \{H^{r-d}(\lambda_z)\})$$
where $H^{r-d}(\lambda_z)$ is the Hecke transformation of $\lambda_z$ (which is the inverse of $H_z$ considered in Remark \ref{rmks3.11}). Note that, even if $|I|=0$, we can add a trivial parabolic structure $\lambda_z=(k, k, ...,k)$ which
does not change the numbers $D_g(r, d, \{\lambda_x\}_{x\in I})$ and $V_g(r, d, \{\lambda_x\}_{x\in I})$. Thus, to finish our proof, we only need to show
\ga{4.19}{V_g(r,d,\{\lambda_x\}_{x\in I})=V_g(r, 0, \{\lambda_x\}_{x\in I'}\cup \{H^{r-d}(\lambda_z)\}).}
Recall that for any $\mu=(\mu_1,\cdots, \mu_r),$ $0\leq \mu_r\leq \cdots \leq \mu_1\leq k$, we define
$$H^1(\mu)=(k-\mu_{r-1}+\mu_r,\mu_1-\mu_{r-1},\mu_2-\mu_{r-1},\cdots, \mu_{r-2}-\mu_{r-1},0)$$
and $H^{m}(\mu):=H^1(H^{m-1}(\mu))$ for $2\leq m\leq r$. It is enough to show
\ga{4.20}{V_g(r,d,\{\lambda_x\}_{x\in I})=V_g(r, d+1, \{\lambda_x\}_{x\in I'}\cup \{H^1(\lambda_z)\}).}
Note that $|H^1(\mu)|=k-r\mu_{r-1}+|\mu|$ and $S_{H^1(\mu)}\left( \mathrm{exp}2\pi i\frac{\vec v}{r+k}\right)=$
$$(-1)^{r-1}\mathrm{exp}2\pi i\left(-\frac{\mu_{r-1}+1}{r+k}\sum_{i=1}^r v_i\right)
S_{\mu}\left(\mathrm{exp}2\pi i\frac{\vec v}{r+k}\right).$$
It is easy to check \eqref{4.20} and we are done.
\end{proof}

\bibliographystyle{plain}

\renewcommand\refname{References}

\end{document}